\documentclass[onecolumn,noversion,nonote]{cdcarticle}

\usepackage{amssymb}
\usepackage{amsmath}
\usepackage{enumerate}
\usepackage{amsthm}
\usepackage{etoolbox}
\usepackage{subfig}
\usepackage{makeidx}
\usepackage{tikz}

\AtBeginEnvironment{align}{\setcounter{subeqn}{0}}
\newcounter{subeqn} \renewcommand{\thesubeqn}{\theequation\alph{subeqn}}%
\newcommand{\subeqn}{%
  \refstepcounter{subeqn}
  \tag{\thesubeqn}
	}

\theoremstyle{definition}
\newtheorem{definition}{Definition}

\newtheorem{theorem}{Theorem}

\newtheorem{lemma}{\it Lemma}

\theoremstyle{remark}
\newtheorem{remark}{Remark}

\newcommand{\btheta}{\boldsymbol{\theta}}
\newcommand{\bTheta}{\boldsymbol{\Theta}}

\newcommand{\bx}{\textbf{x}}
\newcommand{\bX}{\textbf{X}}
\newcommand{\f}{\textbf{f}}
\newcommand{\bu}{\textbf{u}}
\newcommand{\bU}{\textbf{U}}
\newcommand{\bD}{\textbf{D}}

\newcommand{\be}{\hat{\textbf{e}}}

\begin{document}


\title{Concurrent Learning Adaptive Model Predictive Control
			 with Pseudospectral Implementation}
			
\author{Olugbenga Moses Anubi}

\note{blank}
\maketitle

	%
	%
	%
	%
	%

\begin{abstract}
This paper presents a control architecture in which a direct adaptive control technique is used within the model predictive control framework, using the concurrent learning based approach, to compensate for model uncertainties.  At each time step, the control sequences and the parameter estimates are both used as the optimization arguments, thereby undermining the need for switching between the learning phase and the control phase, as is the case with hybrid-direct-indirect control architectures. The state derivatives are approximated using pseudospectral methods, which are vastly used for numerical optimal control problems. Theoretical results and numerical simulation examples are used to establish the effectiveness of the architecture.
\end{abstract}


\section{Introduction}\label{Introduction}
Model predictive control (MPC) refers to a class of control systems in which the current control action is obtained at each sampling instant by solving a finite(or infinite) horizon open-loop optimal control problem, using the current state of the system as the initial condition. While the result of the optimization is a sequence of control actions over the prediction horizon, only the first control action is applied at the current time; the process is repeated at the next time instant. Using this framework, it is easy and straightforward to cope with hard constraints on controls and states. As a result, MPC has received a lot of attention in the literature for both discrete and continuous time systems \cite{quasi_infinite_MPC1,quasi_infinite_MPC2,quasi_infinite_MPC3,quasi_infinite_MPC4,MPC5,MPC6,MPC8,MPC9,MPC10}. However, due to the dependence on dynamic predictive model, unaccounted modeling errors and dynamic uncertainties may render such model obsolete or inaccurate. In which case, the performance of the MPC can no longer be guaranteed. To overcome this challenge, a number of researchers have proposed some indirect-adaptive MPC approaches which allows for a way to incorporate learning in the MPC framework \cite{adaptiveMPC1,adaptiveMPC2,adaptiveMPC3,adaptiveMPC4}. Using these approaches, the system parameters are estimated online and open-loop optimal controllers are generated at each time step. One major challenge of this approach, however, is that it is difficult to guarantee stability, especially during parameter estimation transient phases \cite{mayne2000nonlinear}.

On the other hand, Direct adaptive control techniques modulate the system input to compensate for modeling uncertainties. Direct adaptive control can guarantee stability, even during harsh transients, however, they do not offer any long-term improvement due to model learning unless the system states are persistently exciting\footnote{A bounded vector signal $\Phi(t)$ is persistently exciting if for all $t>t_0$  there exists $T>0 \text{ and } \gamma>0$ such that $\int_t^{t+T}{\Phi(\tau)\Phi(\tau)^Td\tau}\ge\gamma\textbf{I}$.}. Furthermore, it is difficult to generate optimal solutions in the presence of input and state constraints with direct adaptive architectures \cite{chowdhary2013concurrent}.

 In \cite{chowdhary2013concurrent}, a Concurrent Learning based approach was proposed to address the above challenges. Concurrent learning (CL) \cite{Concurrent_Learning1, Concurrent_Learning2} uses recorded and current data concurrently to learn the parametric uncertainties in a dynamic system. Although it was first introduced for adaptation in the framework of Model Reference Adaptive Control (MRAC)  \cite{Concurrent_Learning1}, it can, as a result of the form of  it's update law, be easily extended to the general framework of adaptive control with linear-in-the-parameter (LP) structure. It was shown\cite{Concurrent_Learning1} that provided that the recorded data satisfies certain rank condition, then the adaptive weight convergence can occur without the system states being persistently exciting.

In this paper, a direct adaptive technique is used within the MPC framework, in conjunction with the Concurrent Learning based approach, to compensate for model uncertainties. At each time step, the control sequences and the parameter estimates are both used as the optimization arguments, thereby undermining the need for switching between the learning phase and the control phase, as is the case with hybrid-direct-indirect control architectures \cite{duarte1989combined,lavretsky2009combined} employed in \cite{chowdhary2013concurrent}. Moreover, the state derivatives are approximated at the recorded data points and over the prediction horizon using pseudospectral method.  Pseudospectral methods are vastly used in the numerical solution of optimal control problems \cite{pseudospectral1,pseudospectral2,pseudospectral3,pseudospectral4,pseudospectral5}. They belong to a class of direct collocation methods where the optimal control problem is transcribed to a nonlinear programming problem (NLP) by parameterizing the state and control using global polynomials, and collocating the differential-algebraic equations using nodes obtained from a Gaussian quadrature. With this approach, it is easier to formulate the problem without requiring that the system dynamics be linearly parameterizable.

The rest of the paper is organized as follows: In Section~\ref{Notation}, the notations used throughout the paper are introduced. In Section~\ref{Preliminary}, the CL problem is reformulated as an optimization problem to facilitate its inclusion into MPC framework. In Section~\ref{Main}, the problem setup for the concurrent learning model predictive control is given. In Section~\ref{Pseudospectral}, the pseudospectral implementation is presented. Numerical examples are given in Section~\ref{Numerical}. Conclusion follows in Section~\ref{Conclusion}.

\section{Notation}\label{Notation}
Throughout the paper, the following notations are used: $\mathbb{R}\text{ and }\mathbb{R}_+$ denotes the set of real numbers and positive real numbers respectively. All vectors and vector functions are treated as row vectors; that is, $\bx(\tau) = [x_1(\tau),\hdots,x_n(\tau)]\in\mathbb{R}^n$, where $n$ is the continuous time dimension of $\bx(\tau)$. The Euclidean norm of a vector $\bx\in\mathbb{R}^n$ is denoted by $\left\|\bx\right\|\triangleq\left(\bx^T\bx\right)^{1/2}$. The quadratic form $\left\|\bx\right\|_P^2\triangleq\bx^TP\bx$ is defined for any symmetric positive semi-definite matrix $P$. The expression $P\preceq Q$ means that the matrix $P-Q$ is negative semi-definite. The transpose of a matrix $\textbf{B}$ is denoted by $\textbf{B}^T$. The $i$th row of a matrix $\textbf{D}$ is denoted by $\textbf{D}_i$ .  The gradient of a scalar valued function $f(\bX)$ is a row vector denoted by $\nabla f(\bX)$.  The matrix $\f_\theta(\hdots,\theta_H,\dots)\in\mathbb{R}^{p\times n}$ denotes the partial derivative of a vector valued function $\f(\hdots,\theta,\hdots):\hdots\times\mathbb{R}^p\times\hdots\mapsto\mathbb{R}^n$ with respect to the argument $\theta\in\mathbb{R}^p$, evaluated at $\theta = \theta_H$. The $\mathcal{L}_2$ norm of a vector-valued signal $\bx: \mathbb{R}_+\mapsto\mathbb{R}^n$ is given by
\begin{align*}
\left\|\bx\right\|_2  = \left(\int_0^\infty\left\|\bx(\tau)\right\|^2d\tau\right)^{1/2}
\end{align*}$\mathcal{H}_n^{\alpha}$ denotes the $n$-vector valued Sobolev space over the interval $[-1,\hspace{2mm}1]$, with $\alpha$ denoting the number of classical derivatives of its elements. $\mathcal{H}^{\alpha}$ is given with respect to the $\mathcal{L}_2$ norm as,
\begin{align*}
\left\|x\right\|_{\left(\alpha\right)}=\left(\sum_{k=0}^\alpha{\left\|x^{\left(k\right)}\right\|_2^2}\right)^{1/2}.
\end{align*}
The space of all bounded functions is denoted by $\mathbb{L}_\infty$.
\section{Concurrent Learning}\label{Preliminary}
In this section, the original CL problem is reformulated as an optimization problem. This facilitates a direct inclusion into the MPC framework, as will be shown in the next section.
The class of system considered is described by the following set of nonlinear ordinary differential equations:

\begin{align}\label{sys}
\dot{\bx}(t)&=\f(\bx(t),\bu(t),\btheta), \hspace{5mm}\bx(0)=\bx_0,
\end{align}
where $\bx(t)\in\mathbb{R}^n$ is the vector of state variables, $\bu(t)\in\mathbb{R}^m$ is a vector of inputs, and $\btheta\in\bTheta\subset\mathbb{R}^p$ is a vector of unknown constant parameters. The following assumptions  are made for the system described in \eqref{sys}(see \cite{quasi_infinite_MPC1} also for a similar set of assumptions):

\renewcommand{\theenumi}{(A\arabic{enumi})}
\begin{enumerate}
\item{$\f:\mathbb{R}^n\times\mathbb{R}^m\times\mathbb{R}^p\rightarrow\mathbb{R}^n$ is twice continuously differentiable and $\f(\textbf{0},\textbf{0},\btheta)=\textbf{0}, \hspace{3mm} \forall \btheta\in\bTheta$. That is, $\textbf{0}\in\mathbb{R}^n$ is an equilibrium of the system with $\bu = \textbf{0}.$}
\item{$\bu(t)\in\mathcal{U}$, where $\mathcal{U}\subset\mathcal{R}^m$ is compact, convex, and $\textbf{0}\in\mathbb{R}^m$ is contained in the interior of $\mathcal{U}$.}
\item{The system in \eqref{sys} has a unique solution for any initial condition $\bx_0\in\mathbb{R}^n$ and any piecewise continuous and right continuous $\bu(.):[0,\infty)\rightarrow\mathcal{U}$, for all $\btheta\in\bTheta$.}
\end{enumerate}
Let
\begin{align}\label{eq:history}
\Omega(\bu_H(.))=\left\{\bx_H:\dot{\bx}_H-\f(\bx_H,\bu_H,\btheta)=\textbf{0}\right\}
\end{align}
be a set of recorded data generated by the system in \eqref{sys} from a given open-loop input sequence $\bu_H(\tau_H),\hspace{2mm}\tau_H\in[0,\hspace{2mm}T]$, and unknown constant parameter $\btheta$.

The following definition of persistence of excitation is adopted for the subsequent development in this paper.

\begin{definition}[Persistence of Excitation (PE)]
The system in \eqref{sys} is said to be persistently exciting with respect to the open loop input sequence $\bu(t)$, if there exists $\lambda_1,\lambda_2>0$ such that
\begin{align}\label{eq:PE}
\lambda_1 I\preceq\int_0^T{\f_\theta(\bx_H,\bu_H(\tau_H),\btheta_H)\f_\theta(\bx_H,\bu_H(\tau),\btheta_H)^Td\tau}\preceq\lambda_2 I,
\end{align}
for all $\btheta_H\in\bTheta$.
\end{definition}

Let $\hat{\btheta}$ be an estimate of the unknown parameter $\btheta$, the performance index
\begin{align}\label{eq:PI}
\epsilon(\hat{\btheta}(t))=\int_0^T{\left\|\dot{\bx}_H(\tau_H)-\f(\bx_H(\tau_H),\bu_H(\tau_H),\hat{\btheta}(t))\right\|^2d\tau_H}.
\end{align}
is defined to characterize the ``goodness"\footnote{This term is used to describe how close the response, generated using the estimate, is to the actual recorded data} of the parameter estimate $\hat{\btheta}$. Next, the relationship between $\epsilon(\hat{\btheta})$ and the parameter estimation error is exploited.

\begin{theorem}\label{Thm1}
Suppose the system in \eqref{sys} is persistently exciting with respect to the open-loop input sequence $\bu_H(t)$, then for all $\varepsilon>0$, there exists $\delta>0$, satisfying $\delta\rightarrow0\text{ as }\varepsilon\rightarrow0$, such that $\left\|\btheta-\hat{\btheta}\right\|\le\delta$ whenever $\epsilon(\hat{\btheta})\le\varepsilon$.
\end{theorem}

\begin{proof}
After using \eqref{eq:history}, the equation in \eqref{eq:PI} can be written as
\begin{align}\label{eq:MVT1}
\epsilon(\hat{\btheta})=\int_0^T{\left\|\f(\bx_H,\bu_H(\tau),\btheta)-\f(\bx_H,\bu_H(\tau),\hat{\btheta})\right\|^2d\tau}.
\end{align}
It is clear to see, using the Mean Value Theorem, that
\begin{align}
\f(\bx_H,\bu_H(\tau),\btheta)-\f(\bx_H,\bu_H(\tau),\hat{\btheta})=\left(\btheta-\hat{\btheta}\right)\f_\theta(\bx_H,\bu_H(\tau),\btheta_H),
\end{align}
where
\begin{align*}
\btheta_H = \alpha\btheta+(1-\alpha)\hat{\btheta},\hspace{3mm}\alpha\in[0,\hspace{2mm}1].
\end{align*}
Thus, the equation in \eqref{eq:MVT1} becomes
\begin{align}\label{eq:MVT2}
\epsilon(\hat{\btheta})=\left(\btheta-\hat{\btheta}\right) \left(\int_0^T{\f_\theta(\bx_H,\bu_H(\tau),\btheta_H)\f_\theta(\bx_H,\bu_H(\tau),\btheta_H)^Td\tau}\right)\left(\btheta-\hat{\btheta}\right)^T.
\end{align}
Now, using \eqref{eq:PE} , it follows that
\begin{align}
\lambda_1\left\|\btheta-\hat{\btheta}\right\|^2\le\epsilon(\hat{\btheta}).
\end{align}
Thus, the conclusion follows by setting 
\begin{align*}
\delta = \frac{\sqrt{\varepsilon}}{\lambda_1}.
\end{align*}
\end{proof}

\begin{remark}
The model error given by the performance index in \eqref{eq:PI} requires the computation of the state derivatives. This can be computed accurately using numerical smoothing techniques \cite{Concurrent_Learning2,Concurrent_Learning4}. However, as will be shown in subsequent sections, the need to compute state derivatives is abated by transforming the problem using pseudospectral approximation.
\end{remark}

\begin{remark}
Theorem \ref{Thm1} shows that the smaller the value of the performance index in \eqref{eq:PI}, the smaller the 2-norm of the parameter estimation error. Thus, the parameter estimation error can be reduced as much as possible by setting
\begin{align*}
\hat{\btheta} = \arg\min_{\btheta'}\epsilon(\btheta').
\end{align*}

\end{remark}

\section{Concurrent Learning Adaptive Model Predictive Control Scheme}\label{Main}
In this section, the problem setup for the concurrent learning adaptive model predictive control is given. Following each measurement, an open-loop optimal control is solved. The objective function to minimize comprises of the performance index given in \eqref{eq:PI}, and an additional cost functional which penalizes the state and control in accordance with standard MPC setup. Moreover, the arguments of the optimization is the pair $(\bu(.),\bar{\btheta})$. In other words, at each time step, the values of the open-loop control sequence $\bu(.)$ and a constant parameter estimate $\bar{\btheta}$ that minimizes the combined cost functional is found. In particular, the open-loop optimal control problem at time $t$, with initial state $\bx(t)$, is formulated as

\begin{align}\label{opt}
\min_{(\bar{\bu}(.),\bar{\btheta})}J(\bx(t),\bar{\bu}(.),\bar{\btheta}),
\end{align}
where
\begin{align}\label{obj}
J(\bx(t),\bar{\bu}(.),\bar{\btheta}) = \int_t^\infty{\left(\left\|\bar{\bx}(\tau,\bx(t))\right\|_Q^2+\left\|\bar{\bu}(\tau)\right\|_R^2\right)d\tau} + \gamma\epsilon(\bar{\btheta}),
\end{align}
subject to
\begin{align}\label{state1}
&\dot{\bar{\bx}}=\f(\bar{\bx},\bar{\bu},\bar{\btheta}),\hspace{3mm}\bar{\bx}(t,\bx(t))=\bx(t) \refstepcounter{equation}\subeqn\\\label{const1}
&\bar{\bu(\tau)}\in\mathcal{U},\hspace{3mm}\tau\in[t,\hspace{2mm}\infty)\subeqn,
\end{align}
where $\gamma>0$, and $Q\in\mathbb{R}^{n\times n}$ and $R\in\mathbb{R}^{m\times m}$ are positive definite symmetric weighting matrices; $\bar{\bx}(\tau,\bx(t))$ is the state trajectory of the system in \eqref{state1}, starting from the initial state $\bx(t)$, and driven by the open-loop control sequence $\bu(\tau), \tau\in[t,\hspace{2mm}\infty)$. Without loss of generality, an infinite-horizon nonlinear model predictive control problem is considered. For a finite-horizon\footnote{Interested readers are directed to references \cite{quasi_infinite_MPC1,quasi_infinite_MPC2,quasi_infinite_MPC3,quasi_infinite_MPC4}.} case, the problem can be setup to include an additional quadratic terminal cost chosen to ensure that a closed-loop asymptotic stability is guaranteed. 

According to the receding horizon philosophy, the resulting open-loop optimal control profile is applied to the system only until the next measurement becomes available. Let $T_s$ be the measurement sampling time, and $(\bar{\bu}^*(\tau,\bx(t)),\bar{\btheta}^*(\bx(t))$ the optimal solution to the optimization problem \eqref{opt}--\eqref{const1}, then the closed-loop control and parameter estimate are given by

\begin{align}\label{eq:clp_u}
\bu(\tau)&=\bar{\bu}^*(\tau,\bx(t)),\hspace{2mm}\tau\in[t,\hspace{2mm}t+T_s]\\\nonumber
\hat{\btheta}(\tau) &=\bar{\btheta}^*(\bx(t)) \\\label{eq:clp_theta}&\hspace{3mm}+ k_\theta^{-1} \int_t^\tau{\int_0^T{\left(\dot{\bx}_H-\f(\bx_H,\bu_H(\tau_H),\hat{\btheta}(t))\right)\Gamma(\tau_H)^Td\tau_H} d\sigma},
\end{align}
within the time interval $\tau\in(t,\hspace{2mm}t+T_s)$. $\hat{\btheta}=\bar{\btheta}^*(\bx(\tau))\text{ for } \tau\in\left\{0,T_s,2T_s,\hdots\right\}$. Moreover, $\Gamma :\mathbb{R}_+\rightarrow \mathbb{R}^{n\times p}$ is  chosen to satisfy
\begin{align}\label{eq:iniq}
\lambda_1 I\preceq\int_0^T{\f_\theta(\bx_H,\bu_H(\tau_H),\btheta_H)\Gamma(\tau_H)^Td\tau_H}\preceq\lambda_2 I,
\end{align}
for all $\btheta_H\in\bTheta$, and $k_\theta$ is a positive constant, with $k_\theta^{-1}$ being the concurrent learning gain. It should be noted that the optimal parameter $\bar{\btheta}^*$ is merely a decision variable internal to the optimization problem in \eqref{opt}--\eqref{const1}. The update law in \eqref{eq:clp_theta} is given as a by product of the proposed method. The asymptotic convergence to the true parameter, using the given update law, is shown subsequently. Knowledge of the true parameter in the system can be used for several purposes as desired by the user. For instance, diagnostic purposes, as a means to switch between controllers, etc. Once new measurements become available (after $T_s$ time units), the optimization problem in \eqref{opt}--\eqref{const1} is solved again to find new input profiles, the closed-loop control and parameter estimate in \eqref{eq:clp_u} and \eqref{eq:clp_theta} are then applied within the time interval  $\tau\in[t+T_s,\hspace{2mm}t+2T_s],$ and so on. Note that, while the parameter estimate at time $t$ is a function of the state measurement $\bx(t)$, it is treated as a constant throughout the prediction window in the optimization problem in \eqref{opt}--\eqref{const1}. This applies also to all other measurement points $t+kT_s,\hspace{2mm}k = 1,2,\hdots$. Consequently, the closed-loop system is described by the ordinary differential equation
\begin{align}\label{eq:clp}
\dot{\bx}(t) = \f(\bx(t),\bu(t),\hat{\btheta}(t)).
\end{align}
Next, in the following subsection, the stability properties of the closed-loop system is considered.

\subsection{Stability Analysis}
The following standard definitions, adapted from \cite{khalil_book}, describe the notion of stability as used in this paper.

\begin{definition}[Stability]\label{def1}
The equilibrium point $\bx=\textbf{0}$  of the system in \eqref{sys} is stable if for each $\varepsilon>0$ there exists $\eta(\varepsilon)>0$, such that $\left\|\bx(0)\right\|<\eta(\varepsilon)$ implies that $\left\|\bx(t)\right\|<\varepsilon$ for all $t\ge0$.
\end{definition}

\begin{definition}[Asymptotic Stability]\label{def2}
The equilibrium point $\bx=\textbf{0}$  of the system in \eqref{sys} is asymptotically stable if it is stable and $\eta$ can be chosen such that $\left\|\bx(0)\right\|<\eta$ implies that $\bx(t)\rightarrow\textbf{0}\text{ as }t\rightarrow0$.
\end{definition}

Next, in order to facilitate subsequent stability analysis, an important property of the optimal value function is examined. For simplicity of exposition, except required for clarity, the shorthands
\begin{align*}
J(\bx(t))&\triangleq J(\bx(t),\bar{\bu},\bar{\bx},\bar{\btheta})\\
J^*(\bx(t))&\triangleq J(\bx(t),\bar{\bu}^*,\bar{\bx}^*,\bar{\btheta}^*)
\end{align*}
are used. 

\begin{lemma}\label{lemma1}
Suppose the system in \eqref{sys} is persistently exciting with respect to the open-loop input sequence $\bu_H(t)$. If the $k_\theta$ in \eqref{eq:clp_theta} is chosen to satisfy the sufficient condition
\begin{align}
k_\theta\ge\frac{T_s\lambda_3}{\lambda_4},
\end{align}
where
\begin{gather*}
\lambda_3=\max\left\{\lambda_1^3,\lambda_1^2\lambda_2,\lambda_1\lambda_2^2,\lambda_2^3\right\},
\lambda_4=\min\left\{\lambda_1^2,\lambda_1\lambda_2,\lambda_2^2\right\}\\
\left[\lambda_1\text{ and }\lambda_2\text{ are given in \eqref{eq:PE} and \eqref{eq:iniq}}\right]
\end{gather*}
then the optimal value function $J(\bx(t),\bu(\tau),\hat{\btheta}(\tau))\triangleq J^*(\bx(t))$ satisfies 
\begin{align}
J^*(\bx(s))\le J^*(\bx(t)) - \int_t^s{\left(\left\|\bx(\tau)\right\|^2_Q+\left\|\bu^*(\tau)\right\|^2_R+2\beta\left\|\widetilde{\btheta}(\tau)\right\|^2\right)d\tau},
\end{align}
for all $s\in(t,\hspace{2mm}t+T_s]$, where 
\begin{align}
\beta = \frac{\gamma k_\theta\lambda_4}{2\left(T_s\lambda_2+k_\theta\right)^2},
\end{align}
and $\widetilde{\btheta}(\tau)\triangleq\btheta-\hat{\btheta}(\tau)$ is the parameter estimation error.
\end{lemma}
\begin{proof}
At time $t$, the optimal value function, using the closed-loop control in \eqref{eq:clp_u} and parameter estimate in \eqref{eq:clp_theta}, is given by
\begin{align}
J^*(\bx(t)) =  \int_t^\infty{\left(\left\|\bar{\bx}^*(\tau,\bx(t))\right\|_Q^2+\left\|\bu(\tau)\right\|_R^2\right)d\tau} + \gamma\epsilon(\hat{\btheta}(t)).
\end{align}
Now, for all $s\in(t,\hspace{2mm}t+T_s]$, the value of the objective cost functional in \eqref{eq:PI} is given as:
\begin{align}
J(\bx(s)) &=  \int_s^\infty{\left(\left\|\bar{\bx}^*(\tau,\bx(t))\right\|_Q^2+\left\|\bu(\tau)\right\|_R^2\right)d\tau} + \gamma\epsilon(\hat{\btheta}(s)),\\\nonumber
&=  \int_t^\infty{\left(\left\|\bar{\bx}^*(\tau,\bx(t))\right\|_Q^2+\left\|\bu(\tau)\right\|_R^2\right)d\tau} + \gamma \epsilon(\hat{\btheta}(s)) \\&\hspace{2cm}-\int_t^s{\left(\left\|\bar{\bx}^*(\tau,\bx(t))\right\|_Q^2+\left\|\bu(\tau)\right\|_R^2\right)d\tau}\\\nonumber
& = J^*(\bx(t)) -  \int_t^s{\left(\left\|\bar{\bx}^*(\tau,\bx(t))\right\|_Q^2+\left\|\bu(\tau)\right\|_R^2\right)d\tau}\\\label{eq:temp3}&\hspace{5cm}+\gamma\left(\epsilon(\hat{\btheta}(s))-\epsilon(\hat{\btheta}(t))\right).
\end{align}
For the sake of clarity, let 
\begin{align*}
\Phi&=\int_0^T{\f_\theta(\bx_H,\bu_H(\tau_H),\btheta_H)\f_\theta(\bx_H,\bu_H(\tau_H),\btheta_H)^Td\tau_H}\\
\Psi&=\int_0^T{\f_\theta(\bx_H,\bu_H(\tau_H),\btheta_H)\Gamma(\tau_H)^Td\tau_H}.
\end{align*}
Thus
\begin{align}
\epsilon(\hat{\btheta}(\tau)) = \widetilde{\btheta}(\tau)\Phi\widetilde{\btheta}(\tau)^T,
\end{align}
which implies that
\begin{align}\nonumber
\epsilon(\hat{\btheta}(s))-\epsilon(\hat{\btheta}(t))&=\left(\widetilde{\btheta}(s)-\widetilde{\btheta}(t)\right)\Phi\left(\widetilde{\btheta}(s)-\widetilde{\btheta}(t)\right)^T \\\label{eq:temp1}&\hspace{3cm}+ 2\widetilde{\btheta}(t)\Phi\left(\widetilde{\btheta}(s)-\widetilde{\btheta}(t)\right)^T.
\end{align}
From \eqref{eq:clp_theta}, we have that
\begin{align*}
\widetilde{\btheta}(s)-\widetilde{\btheta}(t)=k_\theta^{-1} \int_t^s{\int_0^T{\left(\dot{x}_H-\f(\bx_H,\bu_H(\tau_H),\hat{\btheta}(t))\right)\Gamma(\tau_H)^Td\tau_H} d\sigma},
\end{align*}
which, after using the Mean Value Theorem, and following similar argument in \eqref{eq:MVT1} and \eqref{eq:MVT2} yields
\begin{align}\label{eq:temp4}
\widetilde{\btheta}(s)-\widetilde{\btheta}(t)=-k_\theta^{-1} \int_t^s{\widetilde{\btheta}(t)\Psi d\tau_2} = -(s-t)k_\theta^{-1}\widetilde{\btheta}(t)\Psi.
\end{align}
Now, by using the properties in \eqref{eq:PE} and \eqref{eq:iniq}, it is clear that
\begin{align*}
\left.
\begin{array}{r}
\lambda_1I\preceq\Phi\preceq\lambda_2I\\
\lambda_1I\preceq\Psi\preceq\lambda_2I
\end{array}\right\}\Rightarrow
\left\{\begin{array}{l}
\widetilde{\btheta}(t)\Psi\Phi\Psi^T\widetilde{\btheta}(t)^T\le\lambda_3\left\|\widetilde{\btheta}(t)\right\|^2\\
\lambda_4\left\|\widetilde{\btheta}(t)\right\|^2\le\widetilde{\btheta}(t)\Phi\Psi^T\widetilde{\btheta}(t)^T
\end{array}\right..
\end{align*}
Thus, \eqref{eq:temp1} becomes
\begin{align}\nonumber
\epsilon(\hat{\btheta}(s))-\epsilon(\hat{\btheta}(t))&=(s-t)^2k_\theta^{-2}\widetilde{\btheta}(t)\Psi\Phi\Psi^T\widetilde{\btheta}(t)^T-2(s-t)k_\theta^{-1}\widetilde{\btheta}(t)\Phi\Psi^T\widetilde{\btheta}(t)^T\\\nonumber
&\le(s-t)^2k_\theta^{-2}\lambda_3\left\|\widetilde{\btheta}(t)\right\|^2-2(s-t)k_\theta^{-1}\lambda_4\left\|\widetilde{\btheta}(t)\right\|^2\\\nonumber
&=(s-t)k_\theta^{-1}\left((s-t)k_\theta^{-1}\lambda_3-\lambda_4\right)\left\|\widetilde{\btheta}(t)\right\|^2\\&\hspace{5cm}-\int_t^s{k_\theta^{-1}\lambda_4\left\|\widetilde{\btheta}(t)\right\|^2d\tau}.
\end{align}
Now, since $ (s-t)\le T_s$,
\begin{align*}
k_\theta\ge\frac{T_s\lambda_3}{\lambda_4}\Rightarrow k_\theta\ge\frac{(s-t)\lambda_3}{\lambda_4}\Rightarrow(s-t)k_\theta^{-1}\lambda_3-\lambda_4\le0.
\end{align*}
Thus
\begin{align}\label{eq:temp5}
\epsilon(\hat{\btheta}(s))-\epsilon(\hat{\btheta}(t))\le-\int_t^s{k_\theta^{-1}\lambda_4\left\|\widetilde{\btheta}(t)\right\|^2d\tau}.
\end{align}
Moreover, from \eqref{eq:temp4}, we have that
\begin{align}\label{eq:temp5b}
\widetilde{\btheta}(\tau)=\widetilde{\btheta}(t)\left(-(\tau-t)k_\theta^{-1}\Psi+I\right),\hspace{3mm}\tau\in(t,\hspace{2mm}t+T_s]
\end{align}
which implies that
\begin{align}
\left\|\widetilde{\btheta}(\tau)\right\|^2\le\left(T_sk_\theta^{-1}\lambda_2+1\right)^2\left\|\widetilde{\btheta}(t)\right\|^2
\end{align}
Thus, the inequality in \eqref{eq:temp5} yields
\begin{align}\label{eq:temp6}
\epsilon(\hat{\btheta}(s))-\epsilon(\hat{\btheta}(t))\le-\int_t^s{\frac{2\beta}{\gamma}\left\|\widetilde{\btheta}(\tau)\right\|^2d\tau}.
\end{align}
Substituting \eqref{eq:temp6} in \eqref{eq:temp3} yields
\begin{align}
J(\bx(s))\le J^*(\bx(t)) - \int_t^s{\left(\left\|\bx(\tau)\right\|^2_Q+\left\|\bu^*(\tau)\right\|^2_R+2\beta\left\|\widetilde{\btheta}(\tau)\right\|^2\right)d\tau}.
\end{align}
Finally, using the optimality of the value function at $s$, it follows that
\begin{align}
J^*(\bx(s))\le J(\bx(s)),
\end{align}
which implies that
\begin{align*}
J^*(\bx(s))\le J^*(\bx(t)) - \int_t^s{\left(\left\|\bx(\tau)\right\|^2_Q+\left\|\bu^*(\tau)\right\|^2_R+2\beta\left\|\widetilde{\btheta}(\tau)\right\|^2\right)d\tau}.
\end{align*}
\end{proof}

Now, the asymptotic stability result for the closed-loop system in \eqref{eq:clp} is stated in the following theorem.

\begin{theorem}
Suppose that the assumptions (A1)--(A3) are satisfied, also that the sufficient condition and the hypothesis of Lemma~\ref{lemma1} is satisfied, and that the open-loop optimal control problem in \eqref{opt}--\eqref{const1} is feasible for all $t>0$, then the closed-loop system in \eqref{eq:clp}, in the absence of disturbance, with the model predictive control in \eqref{eq:clp_u} and the concurrent learning based update law in \eqref{eq:clp_theta}, is asymptotically stable with asymptotic parameter convergence.
\end{theorem}

\begin{proof} The proof stated here follows similar argument given in \cite{quasi_infinite_MPC1}, with modifications made to include the parameter convergence. First, define the function $V(\bx,\widetilde{\btheta})$ for the closed-loop system in \eqref{eq:clp} as follows:
\begin{align}
V(\bx(t),\widetilde{\btheta}(t))=J^*(\bx(t)) + \int_0^t{\beta\left\|\widetilde{\btheta}(\tau)\right\|^2d\tau}.
\end{align}
Then, $V(\bx,\widetilde{\btheta})$ has the following properties:
\begin{itemize}
\item{$V(\textbf{0},\textbf{0})=0$ and $V(\bx,\widetilde{\btheta})>0$ for $(\bx,\widetilde{\btheta})\neq(\textbf{0},\textbf{0})$, }
\item{$V(\bx,\widetilde{\btheta})$ is continuous at $(\bx,\widetilde{\btheta})=(\textbf{0},\textbf{0})$,}
\item{along the trajectory of the closed-loop system starting from any $\bx_0\in \textbf{X}$, and for $0\le t_1\le t_2\le\infty$
\begin{align}\label{eq:temp7}
V(\bx(t_2),\widetilde{\btheta}(t_2))-V(\bx(t_1),\widetilde{\btheta}(t_1))\le-\int_{t_1}^{t_2}{\left(\left\|\bx(\tau)\right\|^2_Q+\beta\left\|\widetilde{\btheta}(\tau)\right\|^2\right)d\tau.}
\end{align}
}
\end{itemize}
To prove the first property, note that $\widetilde{\btheta}=\textbf{0}\Rightarrow \hat{\btheta}=\btheta$, which , from \eqref{eq:MVT1}, implies that $\epsilon({\btheta})=0$. Thus, It follows from Lemma A.1 in reference \cite{chen1997stability} that $J^*(\textbf{0})=0$. Consequently, $V(\textbf{0},\textbf{0})=0$. Similarly, the second property follows from the continuity of $\f(.,.,.)$ over $\bTheta$, and Lemma A.1 in reference \cite{chen1997stability}. The third property is due to Lemma~\ref{lemma1} and $R>0$. As a result, using standard argument (see \cite{khalil_book}), it can be shown that the equilibrium $(\bx,\widetilde{\btheta})=(\textbf{0},\textbf{0})$ is stable, in accordance with the Definition~\ref{def1}. That is, for each $\varepsilon>0$, there exists $\eta(\varepsilon)>0$, such that $\left\|[\bx(0) \hspace{2mm}\widetilde{\btheta}(0)]\right\|<\eta(\varepsilon)$ implies that $\left\|[\bx(t) \hspace{2mm}\widetilde{\btheta}(t)]\right\|<\varepsilon$ for all $t\ge0$. Moreover, $V(\bx(t),\widetilde{\btheta}(t))\in\mathbb{L}_\infty,\hspace{3mm}\forall t\ge0$, along the closed-loop trajectory.
Next, it will be shown that there exists $\eta>0$ such that $(\bx(t),\widetilde{\btheta})\rightarrow(\textbf{0},\textbf{0})$ as $t\rightarrow\infty$ for all $\left\|[\bx(0) \hspace{2mm}\widetilde{\btheta}(0)]\right\|<\eta$. This implies that the equilibrium $(\bx,\widetilde{\btheta})=(\textbf{0},\textbf{0})$ is asymptotically stable, in accordance with Definition~\ref{def2}.

Starting out with the inequality in \eqref{eq:temp7}, it follows by induction that
\begin{align}
\int_{0}^{\infty}{\left(\left\|\bx(\tau)\right\|^2_Q+\beta\left\|\widetilde{\btheta}(\tau)\right\|^2\right)d\tau}\le V(\bx(0),\widetilde{\btheta}(0))-V(\bx(\infty),\widetilde{\btheta}(\infty)).
\end{align}
Since $V(\bx(\infty),\widetilde{\btheta}(\infty))\ge0$ and $V(\bx(0),\widetilde{\btheta}(0))\in\mathbb{L}_\infty$, it follows that 
\begin{align}
\int_{0}^{\infty}{\left(\left\|\bx(\tau)\right\|^2_Q+\beta\left\|\widetilde{\btheta}(\tau)\right\|^2\right)d\tau}\in\mathbb{L}_\infty,
\end{align}
which further implies that $\int_{0}^{\infty}{\left\|\bx(\tau)\right\|^2_Qd\tau}\in\mathbb{L}_\infty$ and $\int_{0}^{\infty}{\beta\left\|\widetilde{\btheta}(\tau)\right\|^2d\tau}\in\mathbb{L}_\infty$. Thus, $\bx(t),\widetilde{\btheta}\in\mathcal{L}_2$. Furthermore, $\left\|[\bx(t) \hspace{2mm}\widetilde{\btheta}(t)]\right\|\in\mathbb{L}_\infty$, $\mathcal{U}\text{ compact }$, and $\f(.,.,.)$ continuous implies that $\f(\bx(t),\bu(t),\hat{\btheta}(t))\in\mathbb{L}_\infty$ for all $t\in[0,\hspace{2mm}\infty)$. Thus $\bx(t)$ is uniformly continuous. Also, computing the derivative of $\widetilde{\btheta}$ from \eqref{eq:temp5b} using first principle yields
\begin{align}
\dot{\widetilde{\btheta}}(t)\triangleq \lim_{\delta\rightarrow0}\frac{\widetilde{\btheta}(t+\delta)-\widetilde{\btheta}(t)}{\delta}=k_\theta^{-1}\Psi\widetilde{\btheta}(t)\in\mathbb{L}_\infty.
\end{align}
Thus, $\widetilde{\btheta}(t)$ is also uniformly continuous. Consequently, $\left\|\bx(t)\right\|$ and $\left\|\widetilde{\btheta}(t)\right\|$ are uniformly continuous in $t$ on $[0,\hspace{2mm}\infty)$. Thus, it follows from Barbalat's Lemma (\cite{khalil_book}) that
\begin{align}
\left\|\bx(t)\right\|\rightarrow0,\text{ and }\left\|\widetilde{\btheta}(t)\right\|\rightarrow0,\text{ as }t\rightarrow\infty.
\end{align}
\end{proof}

\section{Pseudospectral Implementation}\label{Pseudospectral}
In this section, the open-loop infinite horizon optimal control problem in \eqref{opt}--\eqref{const1} is transcribed into an NLP using pseudospectral method. First, the details of the collocation are given. Then, the resulting NLP for the optimal control problem is formulated. The effect of the pseudospectral approximation on the stability of the resulting closed-loop system is also examined.

The most commonly used sets of collocation points are Legendre-Gauss (LG), Legendre-Gauss-Radau (LGR), and Legendre-Gauss-Lobatto (LGL) points. They are obtained from the roots of a Legendre polynomial and/or linear combinations of Legendre polynomial and its derivatives. All three sets of points are defined on the domain $[-1,\hspace{2mm}1]$, but differ significantly in that the LG points include neither of the endpoints, the LGR points include one of the end points, and the LGL points include both of the endpoints.

The LGR collocation scheme is used for the purpose of this paper. The reason for this is because using the pseudospectral form of the LGR scheme results in a system of equations that has no loss of information from the integral form (this is due to the special form of the resulting differentiation matrix)\cite{pseudospectral5}. For the infinite horizon part of the cost functional in \eqref{obj}, the interval $[-1,\hspace{2mm}1]$ is mapped into $[t,\hspace{2mm}\infty)$ using the change of variable
\begin{align}
\tau = \phi(\nu_1),
\end{align}
where $\phi$ is a differentiable, strictly monotonic function. Three examples of such functions are given, based on the ones given in references \cite{pseudospectral_infinite1,pseudospectral_infinite2}, as

\begin{align}
\phi_a(\nu_1) &= t+\frac{1+\nu_1}{1-\nu_1}\\
\phi_b(\nu_1) &=t+\log\left(\frac{2}{1-\nu_1}\right)\\
\phi_c(\nu_1) &=t+\log\left(\frac{4}{\left(1-\nu_1\right)^2}\right).
\end{align}
For the recorded data (the finite horizon part of the cost functional), the interval $[0,\hspace{2mm}T]$ is mapped into $[-1,\hspace{2mm}1]$ using the affine transformation 
\begin{align}
\tau_H = \frac{T}{2}(\nu_2 +1).
\end{align}
Let $S(\nu_1)\triangleq d\phi/d\nu_1\equiv\phi'(\nu_1)$, then the infinite horizon optimal control problem in \eqref{opt}--\eqref{const1} becomes

\begin{align}\nonumber
\min_{(\bar{\bu}(.),\bar{\bu}(.),\bar{\btheta})}J&(\bx(t),\bar{\bu}(.),\bar{\bx}(.),\bar{\btheta}) \\\nonumber&= \int_{-1}^{+1}{S(\nu_1)\left(\left\|\bar{\bx}(\nu_1)\right\|_Q^2+\left\|\bar{\bu}(\nu_1)\right\|_R^2\right)d\nu_1}\\\label{obj2}&\hspace{1cm} + \frac{\gamma T}{2}\int_{-1}^{+1}{\left\|\dot{\bx}_H(\nu_2)-\frac{T}{2}\f(\bx_H(\nu_2),\bu_H(\nu_2),\bar{\btheta})\right\|^2d\nu_2}
\end{align}
subject to
\begin{align}\label{state2}
&\dot{\bar{\bx}}(\nu_1)=S(\nu_1)\f(\bar{\bx}(\nu_1),\bar{\bu}(\nu_1),\bar{\btheta}),\hspace{3mm}\bar{\bx}(-1)=\bx(t) \refstepcounter{equation}\subeqn\\\label{const2}
&\bar{\bu}(\nu_1)\in\mathcal{U}\subeqn.
\end{align}
Here, $\bar{\bx}(\nu_1)$, $\bar{\bu}(\nu_1)$ and $\bar{\btheta}(\nu_1)$ denote the state, the control and the parameter estimate as a function of the new variable $\nu_1$. The independent variable $\nu_2$ denotes the transformed time variable for the recorded data.

Next, the discrete approximations using LGR pseudospectral scheme is described. Consider the LGR collocation points $-1=\tau_1<\hdots<\tau_N<+1$, and the additional non collocated point $\tau_{N+1}=+1$. The interior of the collocation points are given by the zeros of the derivative of the $N$th-order Legendre polynomial $L_{P_N}(x)$, i.e $\left\{\tau_j\right\}_2^{N-1}\triangleq\left\{x:{L}_{P_N}^{\prime}(x)=0\right\}$ \cite{canuto2006erratum}. The state is then approximated by a polynomial of degree at most $N$ as follows:

\begin{align}\label{eq:state_approx1}
\bx(\nu_1)&\approx\sum_{j=1}^{N+1}{\bX_jL_j(\nu_1)},\\\label{eq:state_approx2}
\bx_H(\nu_2)&\approx\sum_{j=1}^{N+1}{\bX_{H_j}L_j(\nu_2)},\\
L_j(\nu) &= \prod_{\substack{k=1\\k\neq j}}^{N+1}{\frac{\nu-\tau_k}{\tau_j-\tau_k}},\hspace{2mm}j=1,\hdots, N+1,
\end{align}
where $L_j$ is a basis of $N$th-degree Lagrange polynomials. Differentiating the state approximations in \eqref{eq:state_approx1} and \eqref{eq:state_approx2}, and evaluating at the collocation points yields
\begin{align}
\dot{\bx}(\tau_i)&\approx\sum_{j=1}^{N+1}{\bX_j\dot{L}_j(\tau_i)}=\sum_{j=1}^{N+1}D_{ij}{\bX_j}=\bD_iX,\\
\dot{\bx}_H(\tau_i)&\approx\sum_{j=1}^{N+1}{\bX_{H_j}\dot{L}_j(\tau_i)}=\sum_{j=1}^{N+1}D_{ij}{\bX_{H_j}}=\bD_iX_H
\end{align}
where 
\begin{align*}
D_{ij}=\dot{L}_J(\tau_i),X=\begin{bmatrix}\bX_1\\\vdots\\\bX_{N+1}\end{bmatrix}\text{ and }\bX_H=\begin{bmatrix}X_{H_1}\\\vdots\\\bX_{H_{N+1}}\end{bmatrix}.
\end{align*}
The matrix $D\in\mathcal{R}^{N\times(N+1)}$ with entries $D_{ij},\hspace{2mm}(i=1,\hdots,N;j=1,\hdots,N+1)$ is the Radau Pseudospectral Differentiation Matrix, since it transforms the state approximation at the points $\tau_1,\hdots,\tau_{N+1}$ to the derivatives of the state approximation at the LGR points $\tau_1,\hdots,\tau_N$. As result, using this formulation averts the use of any numerical smoothing techniques, otherwise needed to compute the state derivatives for the recorded data. 

It is noted that the matrix $X_H$ is composed of the state approximations of the recorded data at the collocation points only. These are generally unknown, since the recorded data are assumed to be measured at specific points which are generally not the collocation points. As a result, a transformation is needed to express $X_H$ in terms of the measured recorded data $X_H^m\in\mathbb{R}^{N_m\times(N+1)}$, where $N_m$ is the number of measurement points. It is required that $N_m>N$ to ensure that the corresponding measured data maps to a unique set of $X_H$. Let $\nu_2=\sigma_1,\hdots,\sigma_{N_m}$ denote the measurement points for the recorded data, then from \eqref{eq:state_approx2}
\begin{align}
\bx_H(\sigma_k)&\approx\sum_{j=1}^{N+1}{\bX_{H_j}L_j(\sigma_k)},\hspace{2mm}k=1,\hdots,N_m.
\end{align}
Thus,
\begin{align}
X_H^m = M_xX_H,
\end{align}
where the matrix $M_x\in\mathbb{R}^{N_m\times(N+1)}$ has entries $M_{x_{kj}}=L_j(\sigma_k)$. Since $N_m\ge (N+1)$, it follows from the orthogonality of the Legendre polynomials that $\text{rank}(M)=N+1$. As a result
\begin{align}
X_H=M_x^\dagger X_H^m=\left(M_x^TM_x\right)^{-1}M_x^TX_H^m
\end{align}
will yield a unique state approximation data $X_H$ for every unique measured state data $X_H^m$. Similarly, the open-loop control signals at the collocation points are given in terms of the open-loop controls at the measurement points as
\begin{align}
U_H=M_u^\dagger U_H^m=\left(M_u^TM_u\right)^{-1}M_u^TU_H^m,
\end{align}
where the matrix $M_u\in\mathbb{R}^{N_m\times N}$ has entries $M_{u_{kj}}=L_j(\sigma_k)$.

Let $\bU\in\mathcal{R}^{N\times m}$ be a matrix whose $i$th row $\bU_i$ is an approximation to the control $\bu(\tau_i),\hspace{2mm}1\le i\le N$. The discrete approximation to the system dynamics in \eqref{state2} is obtained by evaluating the system dynamics at each collocation point and replacing $\dot{\bx}(\tau_i)$ by its discrete approximation $\bD_iX$. Hence, the discrete approximation to the system dynamics is given by
\begin{align}\label{eq:dis_approx}
\bD_iX=S(\tau_i)\f(\bX_i,\bU_i,\bar{\btheta}),\hspace{2mm}1\le i\le N.
\end{align}
Next, the objective function in \eqref{obj2} is approximated by a Legendre-Gauss quadrature as follows:
\begin{align}
J\approx \sum_{i=1}^N w_i\Biggl(S(\tau_i)\left(\left\|\bX_i\right\|^2_Q+\left\|\bU_i\right\|^2_R\right)\hspace{5cm}\notag\\+\frac{\gamma T}{2}\left\|\bD_iM_x^\dagger X_H^m-\frac{T}{2}\f(\be_i M_x^\dagger X_H^m,\be_i M_u^\dagger U_H^m,\bar{\btheta})\right\|^2\Biggr),
\end{align}
where $\be_i$ is the $i$th row of the identity matrix of appropriate dimension and $w_i$ is the quadrature weight, associated with $\tau_i$, given by \cite{hildebrand1987introduction}
\begin{align}
w_i &= \large\left\{\begin{array}{ll}\frac{1-\tau_i}{\left(NP_{N-1}(\tau_i)\right)^2}&\tau_i\neq-1\\\frac{2}{N^2}&\tau_i=-1\end{array}\right.,
\end{align}
where $P_{N-1}$ is the $(N-1)$th Legendre polynomial. The continuous-time nonlinear infinite-horizon optimal control problem in \eqref{opt}--\eqref{const1} is then approximated by the following NLP:

\begin{align}\nonumber
\min_{(U,X,\bar{\btheta})}\bar{J}&(\bx(t),U,X,\bar{\btheta})\\\nonumber&=\sum_{i=1}^N w_i\Biggl(S(\tau_i)\left(\left\|\bX_i\right\|^2_Q+\left\|\bU_i\right\|^2_R\right)\\\label{eq:NLP_opt}&\hspace{1.75cm}+\frac{\gamma T}{2}\left\|\bD_iM_x^\dagger X_H^m-\frac{T}{2}\f(\be_i M_x^\dagger X_H^m,\be_i M_u^\dagger U_H^m,\bar{\btheta})\right\|^2\Biggr),
\end{align}
\begin{align}\nonumber
\text{subject to}&\\
&\bD_iX-S(\tau_i)\f(\bX_i,\bU_i,\bar{\btheta})=0,\hspace{2mm}1\le i\le N,\refstepcounter{equation}\subeqn\\\
&\bx(t)-\bX_1=\textbf{0},\subeqn\\
&\bU_i\in\mathcal{U}, \hspace{2mm}1\le i\le N\subeqn\\\label{eq:NLP_const}
&\bar{\btheta}\in\bTheta\subeqn.
\end{align}

Let $\bU^*_i,\hspace{2mm}1\le i\le N$ and $\bar{\btheta}^*$ be the solution of the NLP in \eqref{eq:NLP_opt}--\eqref{eq:NLP_const}, then the closed-loop control and parameter update laws in \eqref{eq:clp_u} and \eqref{eq:clp_theta} becomes
\begin{align}\label{eq:clp_u2}
\bu(\tau)&=\sum_{j=1}^N{\bU^*_jL_j(\phi^{-1}(\tau))},\\\label{eq:clp_theta2}
\hat{\btheta}(\tau) &= \bar{\btheta}^*+\frac{(\tau-t)T}{2k_\theta}\sum_{j=1}^N{w_j\Gamma(\tau_j)^T\left(\bD_iM_x^\dagger X_H^m-\frac{T}{2}\f(\be_i M_x^\dagger X_H^m,\be_i M_u^\dagger U_H^m,\bar{\btheta}^*)\right)} .
\end{align}

Also, the PE condition requirement of Lemma~\ref{lemma1} reduces to the rank condition
\begin{align}
\text{rank}\left(\sum_{j=1}^Nw_j\f_\theta(\be_i M_x^\dagger X_H^m,\be_i M_u^\dagger U_H^m,\btheta_H)\f_\theta(\be_i M_x^\dagger X_H^m,\be_i M_u^\dagger U_H^m,\btheta_H)^T\right)=p,
\end{align}
for all $\btheta_H\in\bTheta$. This is consistent with the original work in \cite{Concurrent_Learning1} for the special case with LP assumption.

\subsection{Stability Considerations}
Next, the effect of the pseudospectral approximation on the stability of the system is examined. First, some existing established results on the properties of pseudospectral approximations are provided. From these results, the stability of the closed loop system resulting from the control law in \eqref{eq:clp_u2} is studied. Similar to Section~\ref{Main}, except otherwise required for clarity, the shorthands
\begin{align*}
J(\bx(t))&\triangleq J(\bx(t),\bar{\bu},\bar{\bx},\bar{\btheta})\\
J^*(\bx(t))&\triangleq J(\bx(t),\bar{\bu}^*,\bar{\bx}^*,\bar{\btheta}^*)\\
\bar{J}(\bx(t))&\triangleq\bar{J}(\bx(t),U,X,\bar{\btheta})\\
\bar{J}^*(\bx(t))&\triangleq\bar{J}(\bx(t),U^*,X^*,\bar{\btheta}^*)
\end{align*}
are used. 

\begin{lemma}[Interpolation Error Bounds \cite{canuto2006erratum}, Section 5.4.3]\label{lemma2}
If $\bx\triangleq [x_1,\hdots,x_n]\in \mathcal{H}_n^\alpha$, with $x_i\in \mathcal{H}^\alpha, i=1,\hdots,n$, then there exist $\bX_j = \bx(\tau_j), j = 1,\hdots,N+1$, and $c_1,c_{1_i},c_2,c_{2_i}>0$ such that:
\begin{enumerate}[(a)]
\item{The interpolation error is bounded,
	\begin{align}\nonumber
		\left\|\bx(\tau)-\sum_{j=1}^{N+1}{\bX_jL_j(\tau)}\right\|_2&\le\sum_{i=1}^{n}{\left\|x_i-\sum_{j=1}^{N+1}{X_{ij}L_j(\tau)}\right\|_2}\le\sum_{i=1}^{n}{c_{1_i}N^{-\alpha}\left\|x_i\right\|_{(\alpha)}}\\
		&\le c_1N^{-\alpha}.
	\end{align}
}
\item{The error between the exact derivative and the derivative of the interpolation is bounded,
	\begin{align}\nonumber
		\left\|\dot{\bx}(\tau)-\bD(\tau)X\right\|_2&\le\sum_{i=1}^{n}{\left\|\dot{x}_i-\sum_{j=1}^{N+1}{X_{ij}\dot{L}_j(\tau)}\right\|_2}\le\sum_{i=1}^{n}{c_{2_i}N^{1-\alpha}\left\|x_i\right\|_{(\alpha)}}\\
		&\le c_2N^{1-\alpha},
	\end{align}
	where, $\bD(\tau)=[\dot{L}_1(\tau),\dot{L}_2(\tau),\hdots,\dot{L}_{N+1}(\tau)]$.
}
\end{enumerate}
\end{lemma}

\begin{remark}
It is straightforward to see, using the orthogonality property of the Lagrange interpolation polynomial, that the interpolation error is zero at the collocation points. In other words, the approximation is exact at the interpolation points. As a result, any feasible point of the optimization problem in \eqref{obj2}--\eqref{const2} represents the actual system dynamics at the collocation points and the error due to interpolation elsewhere is governed by Lemma~\ref{lemma2}.
\end{remark}
\vspace{1\baselineskip}

\begin{lemma}[Feasibility, Convergence, and Consistency of pseudospectral approximations \cite{justin2011convergence}]\label{lemma3}
Let $\bar{\bx}^*(\tau)\in\mathcal{H}_n^\alpha,\bar{\bu}^*(\tau)\in\mathcal{H}_m^\alpha\text{ and }\bar{\btheta}^*$ be the solution of the optimal control problem in \eqref{obj2}-\eqref{const2}, and $X^*\text{ and } U^*$, the solution of the corresponding NLP in \eqref{eq:NLP_opt}--\eqref{eq:NLP_const}, then the error in the optimal cost functional due to the pseudospectral approximation can be upper bounded as follows;
\begin{align}
\left|J(\bx(t),\bar{\bu}^*,\bar{\bx}^*,\bar{\btheta}^*)-\bar{J}(\bx(t),U^*,X^*,\bar{\btheta}^*)\right|\le \mu(t)N^{-\alpha},
\end{align}
where $\mu(t)>0$ is bounded with bounded derivatives.
\end{lemma}

\begin{theorem}
Suppose that the assumptions (A1)--(A3) are satisfied, also that the sufficient condition and the hypothesis of Lemma~\ref{lemma1} is satisfied, and that the open-loop optimal control problem in \eqref{opt}--\eqref{const1} is feasible for all $t>0$, then the closed-loop system in \eqref{eq:clp}, in the absence of disturbance, with the model predictive control in \eqref{eq:clp_u2} and the concurrent learning based update law in \eqref{eq:clp_theta2} determined from the solution of the NLP in \eqref{eq:NLP_opt}--\eqref{eq:NLP_const}, is uniformly ultimately bounded. Moreover, the ultimate bound can be made arbitrarily small by the choice of the number of collocation points.
\end{theorem}
\begin{proof}
It has been shown that the feasibility of the open-loop optimal control problem in \eqref{opt}--\eqref{const1} implies the feasibility of the NLP in \eqref{eq:NLP_opt}--\eqref{eq:NLP_const} (See \cite{justin2011convergence}). Using Lemma~\ref{lemma3}, the relationship between the value function of the finite-horizon optimal control problem in \eqref{obj2}-\eqref{const2} and the optimal value of the finite-dimensional NLP in \eqref{eq:NLP_opt}--\eqref{eq:NLP_const} can be expressed as
\begin{align}
\bar{J}^*(\bx(s))=J^*(\bx(s))+\mu_1(s) N^{-\alpha},\hspace{2mm}\mu_1(s),\dot{\mu_1}(s)\in\mathbb{L}_\infty
\end{align}
for all $s\in(t,\hspace{2mm}t+T_s]$. Thus, using Lemma~\ref{lemma1}, it follows that
\begin{align}\nonumber
\bar{J}^*(\bx(s))&\le J^*(\bx(t))- \int_t^s{\left(\left\|\bx(\tau)\right\|^2_Q+\left\|\bu^*(\tau)\right\|^2_R+2\beta\left\|\widetilde{\btheta}(\tau)\right\|^2\right)d\tau} +\mu_1(s) N^{-\alpha}\\\nonumber
&= \bar{J}^*(\bx(t))- \int_t^s{\left(\left\|\bx(\tau)\right\|^2_Q+\left\|\bu^*(\tau)\right\|^2_R+2\beta\left\|\widetilde{\btheta}(\tau)\right\|^2\right)d\tau} \\&\hspace{3.25cm}+(\mu_1(s)+\mu_2(s)) N^{-\alpha},\hspace{2mm}\mu_2(s),\dot{\mu_2}(s)\in\mathbb{L}_\infty,
\end{align}
or
\begin{align}\nonumber
\bar{J}^*(\bx(s))&\le\bar{J}^*(\bx(t))- \int_t^s{\left(\left\|\bx(\tau)\right\|^2_Q+\left\|\bu^*(\tau)\right\|^2_R+2\beta\left\|\widetilde{\btheta}(\tau)\right\|^2\right)d\tau} \\\label{eq:temp8}&\hspace{5.5cm}+\mu(s) N^{-\alpha},\hspace{2mm}\mu(s),\dot{\mu}(s)\in\mathbb{L}_\infty.
\end{align}
Similarly, to the proof of Theorem~\ref{Thm1}, define the function
\begin{align}
V(\bx(t),\widetilde{\btheta}(t)) = \bar{J}^*(\bx(t)) + \int_0^t{\beta\left\|\widetilde{\btheta}(\tau)\right\|^2}d\tau.
\end{align}
Taking the time derivative of $V(\bx(t),\widetilde{\btheta}(t))$ yields
\begin{align}
\dot{V}(\bx(t),\widetilde{\btheta}(t)) &= \lim_{s\rightarrow t}\left(\frac{V(\bx(s),\widetilde{\btheta}(s))-V(\bx(t),\widetilde{\btheta}(t))}{s-t}\right)\\
&=\lim_{s\rightarrow t}\left(\frac{\bar{J}^*(\bx(s))-\bar{J}^*(\bx(t))}{s-t} + \frac{1}{s-t}\int_t^s{\beta\left\|\widetilde{\btheta}(\tau)\right\|^2}d\tau\right),
\end{align}
which, after using \eqref{eq:temp8}, can be upper bounded as
\begin{align}\nonumber
\dot{V}(\bx(t),\widetilde{\btheta}(t))&\le-\lim_{s\rightarrow t}\frac{1}{s-t}\int_t^s{\left(\left\|\bx(\tau)\right\|^2_Q+\left\|\bu^*(\tau)\right\|^2_R+\beta\left\|\widetilde{\btheta}(\tau)\right\|^2\right)d\tau}\\\nonumber&\hspace{6cm}+\dot{\mu}(t) N^{-\alpha}\\
&\le-\lim_{s\rightarrow t}\frac{1}{s-t}\int_t^s{\left(\left\|\bx(\tau)\right\|^2_Q+\beta\left\|\widetilde{\btheta}(\tau)\right\|^2\right)d\tau}+\dot{\mu}(t) N^{-\alpha},
\end{align}
which simplifies\footnote{If $f(t)$ is integrable, then there exists a function $F(t)$ such that $F^\prime(t) = f(t)$. Thus $\lim_{s\rightarrow t}\frac{1}{s-t}\int_t^s{f(\tau)d\tau}=\lim_{s\rightarrow t}\frac{F(s)-F(t)}{s-t}=F^\prime(t)=f(t)$} to 
\begin{align}
\dot{V}(\bx(t),\widetilde{\btheta}(t))&\le-\left\|\bx(t)\right\|^2_Q-\beta\left\|\widetilde{\btheta}(t)\right\|^2+\dot{\mu}(t) N^{-\alpha}\\\label{eq:temp9}
&\le-\left\|\bx(t)\right\|^2_Q-\beta\left\|\widetilde{\btheta}(t)\right\|^2+c N^{-\alpha}
\end{align}
for some $c>0$, since $\dot{\mu}$ is bounded. Thus the state and parameter estimation error are uniformly ultimately bounded \cite{khalil_book}. From \eqref{eq:temp9}, it is clear that the ultimate bound can be made arbitrarily small by choosing $N$ appropriately.
\end{proof}

\section{Numerical Example}\label{Numerical}
The following numerical examples are given to demonstrate the proposed control method.
\subsection{Example 1}
Consider a system described by the following ODEs:
\begin{align}
\begin{array}{rl}
\dot{x}_1 &= \left(\theta_1 + |\theta_2x_1|\right)x_2 + u,\\
\dot{x}_2 &= \theta_2x_1.
\end{array}
\end{align}
Here,
\begin{align}
\f_\theta = \left[\begin{array}{cc}x_2&0\\\text{sgn}(\theta_2x_1)x_1x_2&x_1\end{array}\right],
\end{align}
and $\Gamma(\tau)$ is chosen as
\begin{align}
\Gamma(\tau) = \left[\begin{array}{cc}x_2(\tau)&0\\0&x_1(\tau)\end{array}\right].
\end{align}
Thus, the condition in \eqref{eq:iniq} is satisfied with 
\begin{align}
\begin{array}{rl}
\lambda_1&=\min\left\{\int_0^T{x_{H_1}(\tau_H)^2d\tau_H},\int_0^T{x_{H_2}(\tau_H)^2d\tau_H}\right\}\\\\
\lambda_2&=\max\left\{\int_0^T{x_{H_1}(\tau_H)^2d\tau_H},\int_0^T{x_{H_2}(\tau_H)^2d\tau_H}\right\}.
\end{array}
\end{align}
The recorded data is generated using the open loop control
\begin{align}
u(t) = 0.1\sin(5t) + 0.05\cos(2t),
\end{align}
which results in the values of $\lambda_1 = 0.0021$ and $\lambda_2 = 0.0155$. The measurement sampling time is set to $T_s = 0.4 s$. As a result, the optimization routine runs for 0.4s until the next measurement is available. Meanwhile, within the interval $\tau\in[t,\hspace{2mm}t+T_s]$, the control algorithm runs in an open loop fashion based on \eqref{eq:clp_u2} and \eqref{eq:clp_theta2}, using the present state estimate and predictions. The inverse learning rate is set to $k_\theta = 5T_s\lambda_3/\lambda_4 = 0.0309$. The number of LGR nodes used is 5, and the size of the recorded data used is $N_m = 50$. 
\begin{figure}[h!]
  \centering
    \includegraphics[width=9cm,height=5cm]{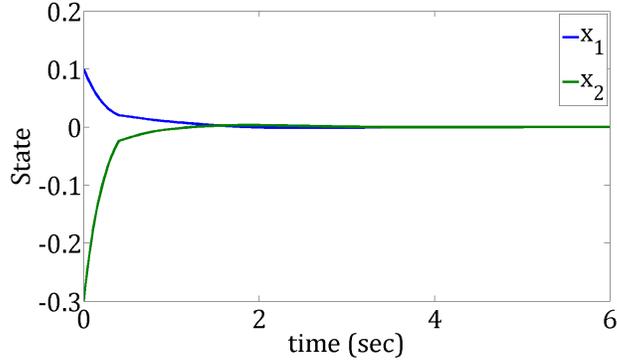} 
 \caption{State trajectory, $T_s = 0.4s$}
\label{fig_state_ex1}
\end{figure}

\begin{figure}[h!]
  \centering
    \includegraphics[width=9cm,height=5cm]{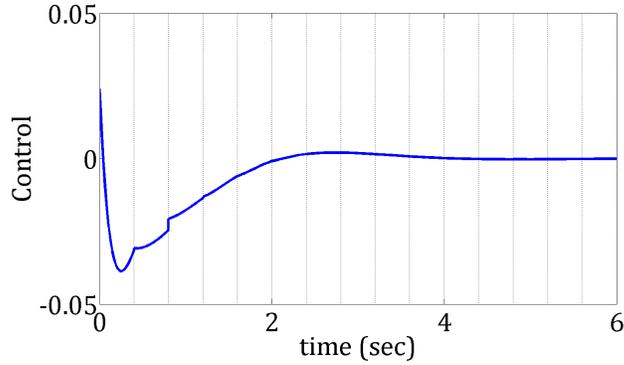}  
\caption{Control trajectory, $T_s = 0.4s$}
\label{fig_control_ex1}
\end{figure}

\begin{figure}[h!] 
  \centering
    \includegraphics[width=9cm,height=5cm]{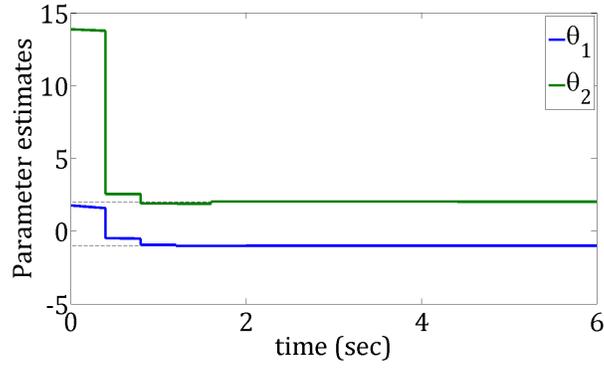}
 \caption{Parameter estimate trajectory, $T_s = 0.4s$}
\label{fig_param_ex1}
\end{figure}

Figure~\ref{fig_state_ex1} shows that the resulting state trajectory converges to the origin asymptotically. The control authority is shown in Figure~\ref{fig_control_ex1}. The faint vertical lines show the measurement points and how the control is updated at those points. Figure~\ref{fig_param_ex1} shows that the parameter estimates converge to the true parameters. 

\begin{figure}[h!]
  \centering
    \includegraphics[width=9.75cm,height=5cm]{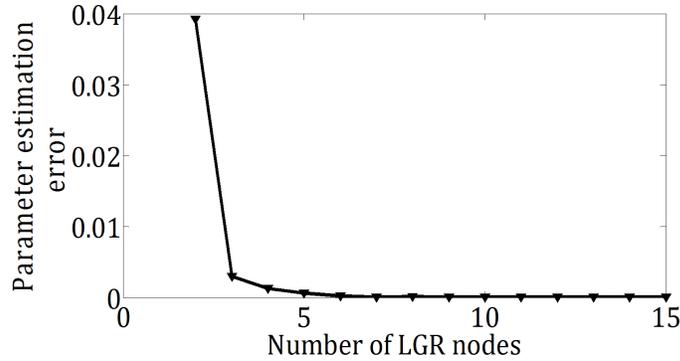}
 \caption{Effect of the number of LGR nodes on parameter estimation}
\label{fig_NLGR}
\end{figure}

As shown in Figure~\ref{fig_NLGR}, the more the number of LGR nodes, the better the ``goodness" of the parameter estimation. This is because a better approximation of the system dynamics is obtained by increasing the number of LGR nodes. As a result, the system parameter are better approximated.

\begin{figure}[h!]
  \centering
    \includegraphics[width=9cm,height=5cm]{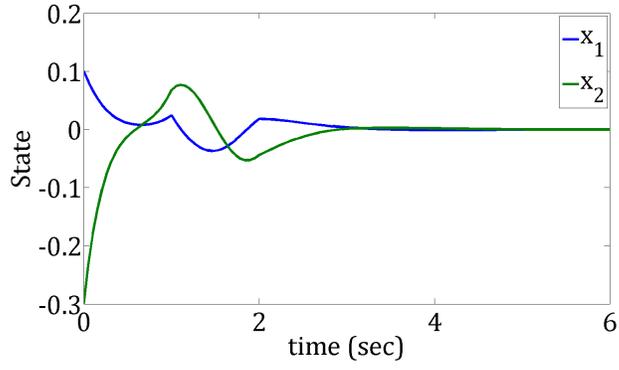} 
 \caption{State trajectory, $T_s = 1s$}
\label{fig_state_ex11}
\end{figure}

\begin{figure}[h!]
  \centering
    \includegraphics[width=9cm,height=5cm]{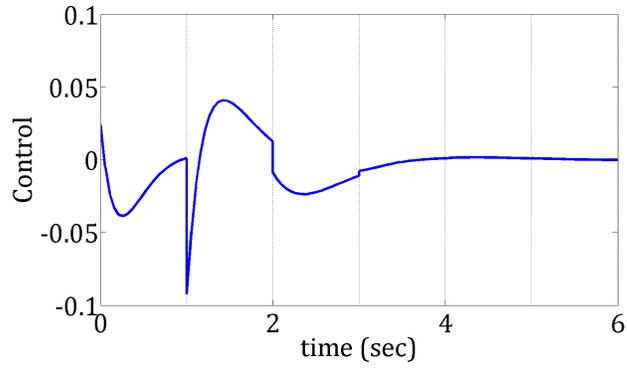}  
\caption{Control trajectory, $T_s = 1s$}
\label{fig_control_ex11}
\end{figure}

\begin{figure}[h!] 
  \centering
    \includegraphics[width=9cm,height=5cm]{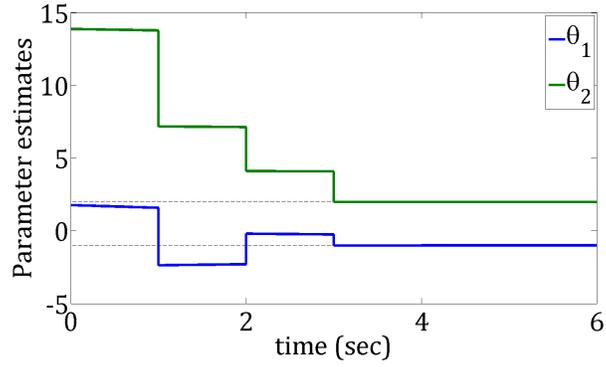}
 \caption{Parameter estimate trajectory, $T_s = 1s$}
\label{fig_param_ex11}
\end{figure}

In order to demonstrate the effect of $T_s$ on the control system, another simulation is carried out with $T_s = 1s$. Figure~\ref{fig_state_ex11} through Figure~\ref{fig_param_ex11} show the resulting state, control and parameter estimate trajectories. It is seen that the parameter estimate, and consequently the control and system response, converges more slowly with increase sampling time.

\subsection{Example 2}
This example demonstrates the special case of linearly parametrized systems. The system considered is a mass-spring-damper system whose dynamics is given by
\begin{align}
\frac{d}{dt}\left[\begin{array}{c}x_1\\x_2\end{array}\right]^T = \left[\begin{array}{c}x_2\\-\frac{k}{m}x_1-\frac{b}{m}x_2+\frac{1}{m}u\end{array}\right]^T,
\end{align}
where $m,k,b$ denote the system mass, spring constant, and damping coefficient values respectively. The dynamics is linearly parametrized as follows
\begin{align}
\frac{d}{dt}\left[\begin{array}{cc}x_1 & x_2\end{array}\right] = \left[\begin{array}{cccc}\theta_1&\theta_2&\theta_3&\theta_4\end{array}\right]\left[\begin{array}{cc}x_2&0\\0&-x_1\\0&-x_2\\0&u\end{array}\right],
\end{align}
where the unknown parameters are given by $\theta_1=1, \theta_2=k/m, \theta_3 = b/m$, $\theta_4 = 1/m$, where $m=2kg, k = 5Nm, b = 0.8Ns/m$. Two simulations were carried out; one in which the control is unconstrained, and the other in which the constraint $|u|\le 0.5$ is imposed on the control authority. Figures~\ref{fig_state_ex2} through~\ref{fig_param_ex2} show the states trajectory, control authority and the parameter updates. As expected, it is seen that the settling time for the constrained case is longer than the unconstrained case. Note that, in this example, the number of unknown parameters is more than the number of states. 

\begin{figure}[h!]
  \centering
   \subfloat[][Unconstrained]{ \includegraphics[width=7cm,height=4cm]{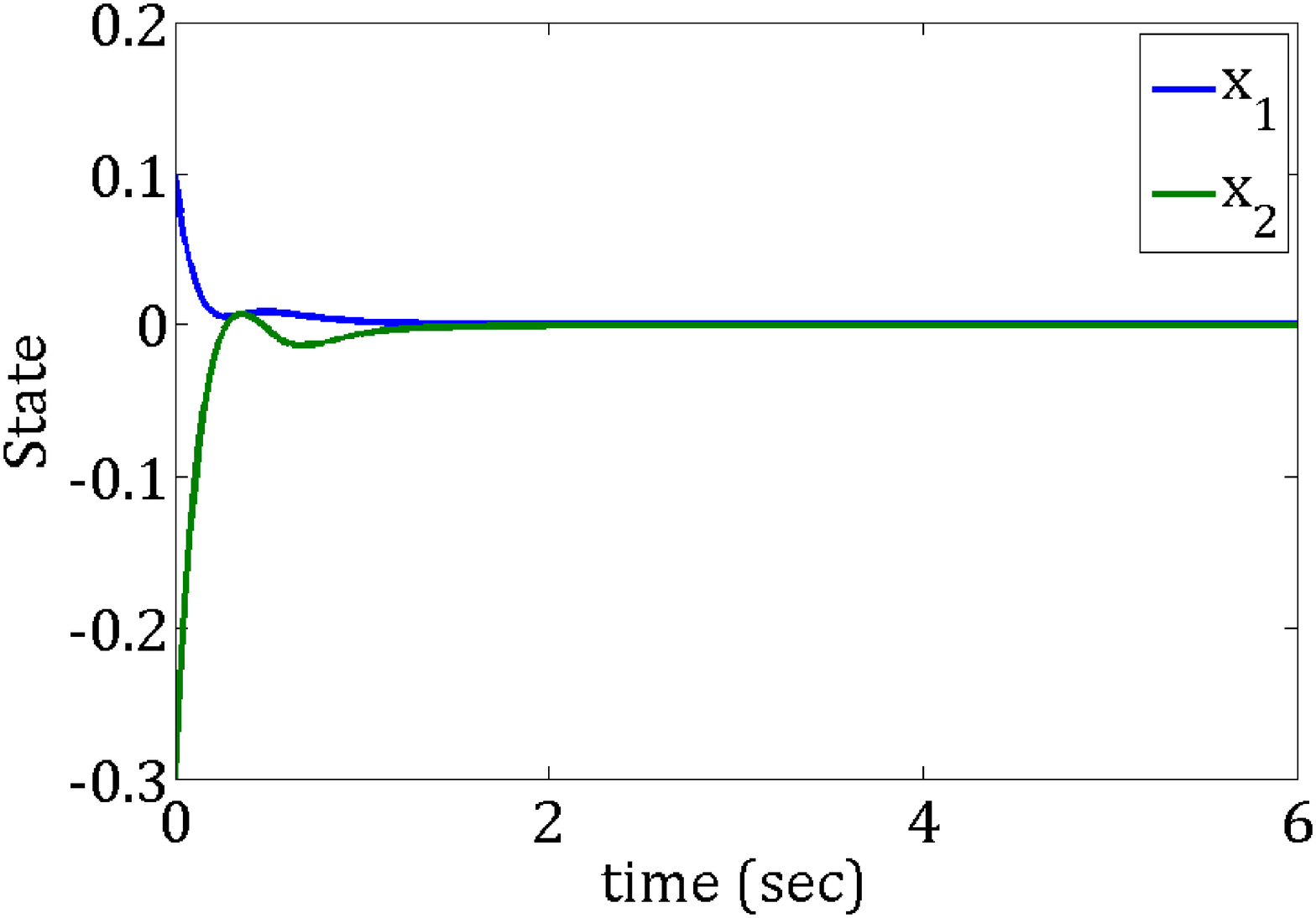}}
   \subfloat[][Constrained]{ \includegraphics[width=7cm,height=4cm]{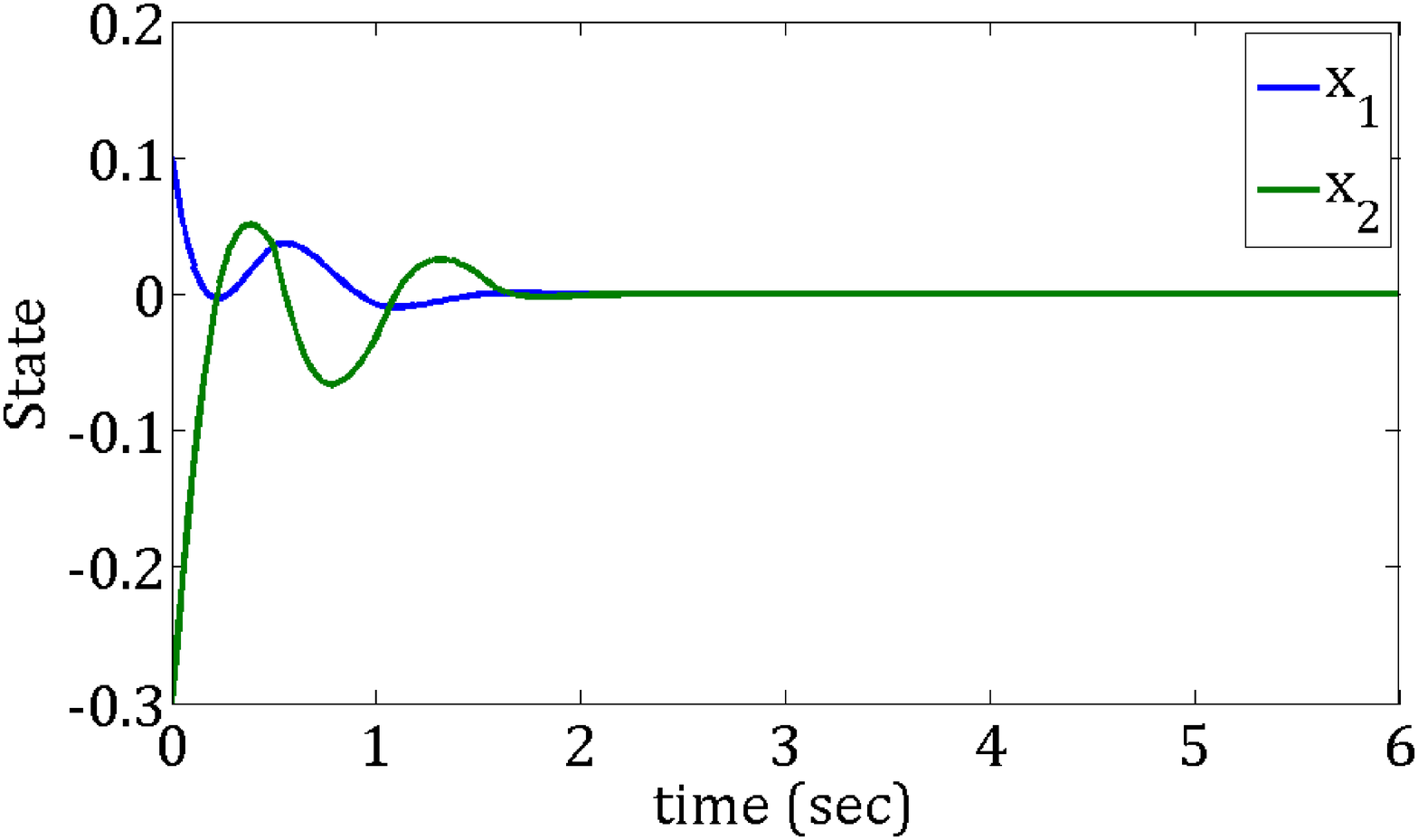}}
 \caption{State trajectory}
\label{fig_state_ex2}
\end{figure}

\begin{figure}[h!]
  \centering
    \subfloat[][Unconstrained]{\includegraphics[width=7cm,height=4cm]{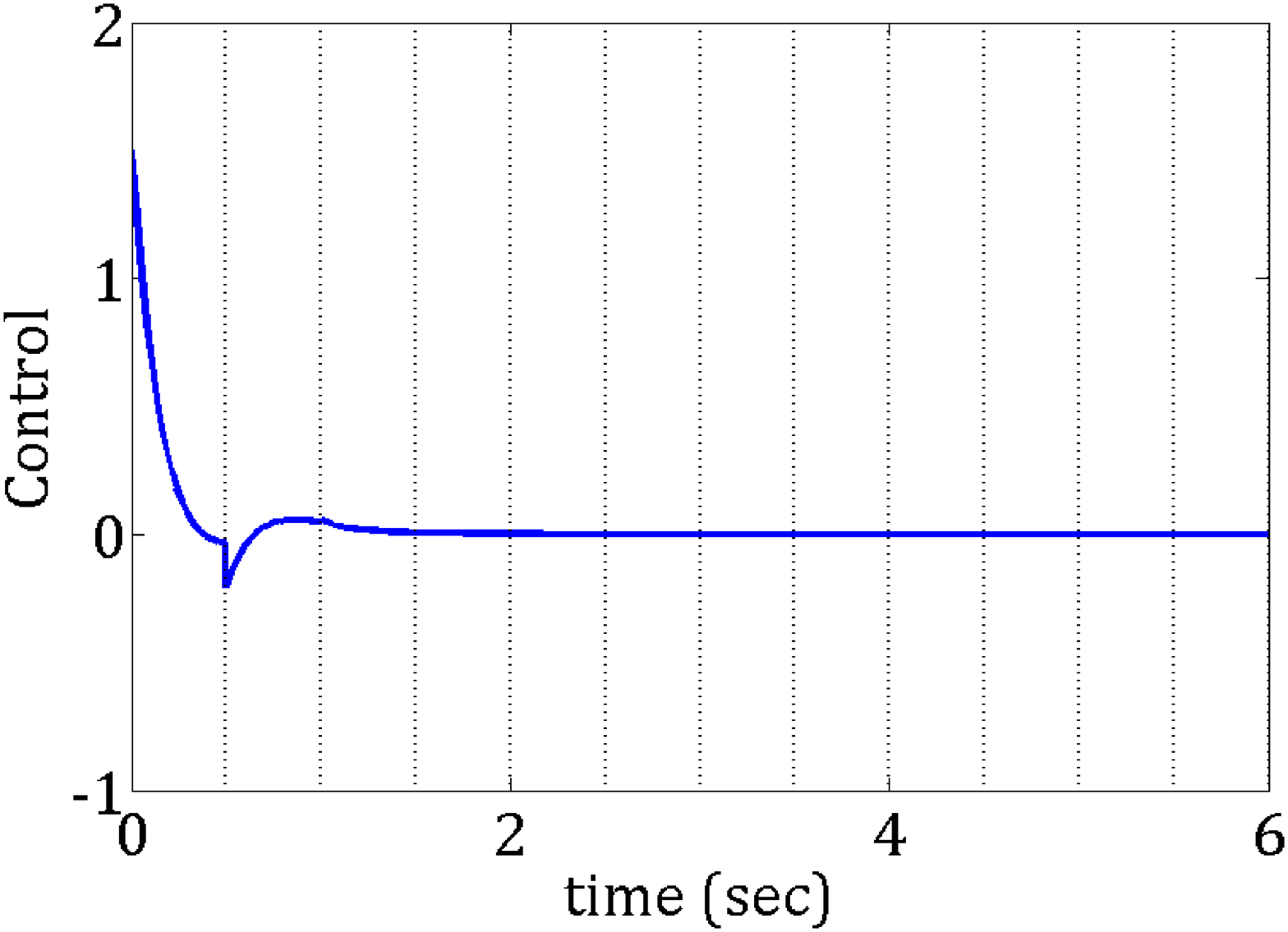} }
    \subfloat[][Constrained]{\includegraphics[width=7cm,height=4cm]{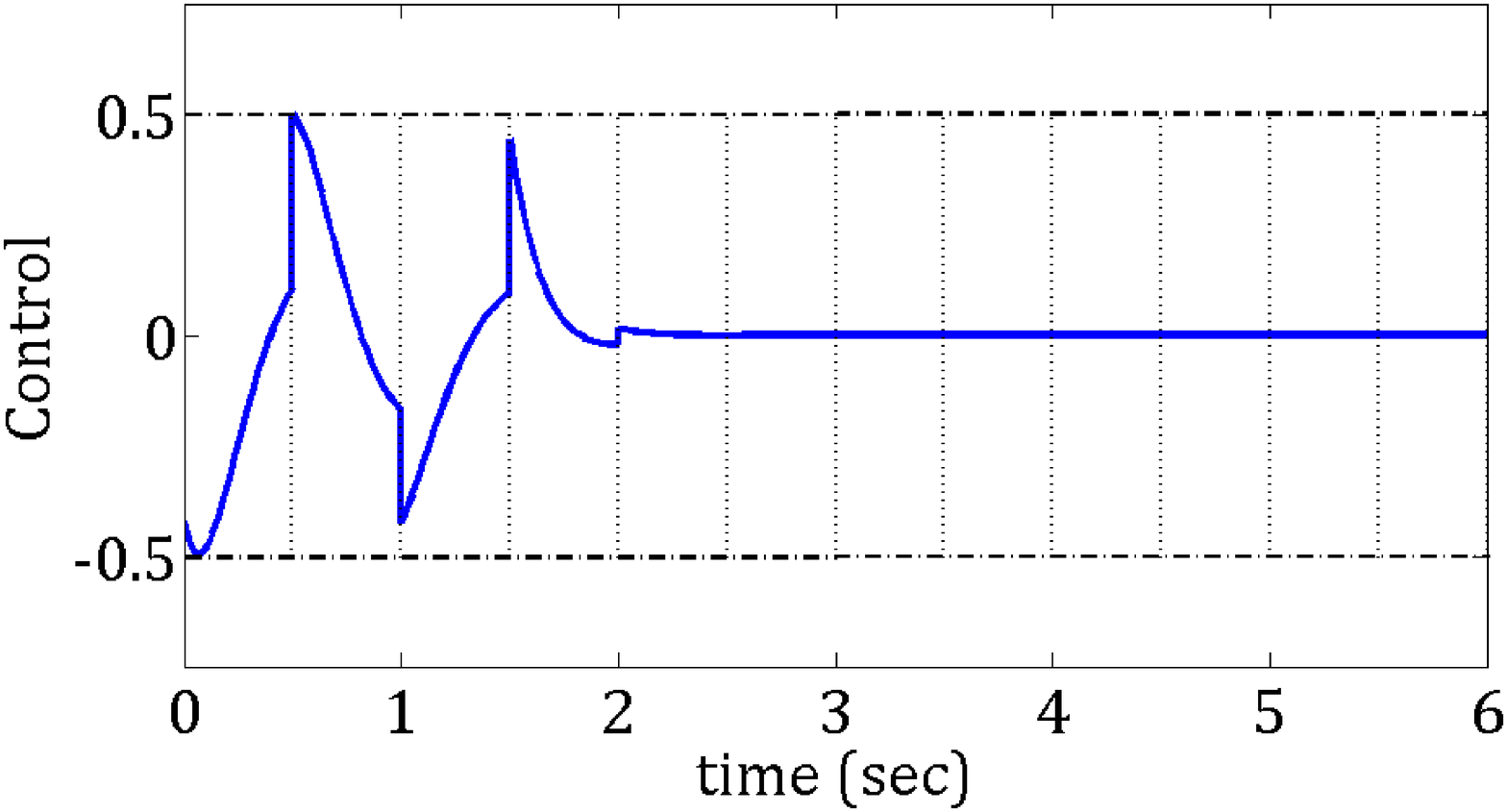} }
\caption{Control trajectory}
\label{fig_control_ex2}
\end{figure}

\begin{figure}[h!] 
  \centering
    \subfloat[][Unconstrained]{\includegraphics[width=7cm,height=4cm]{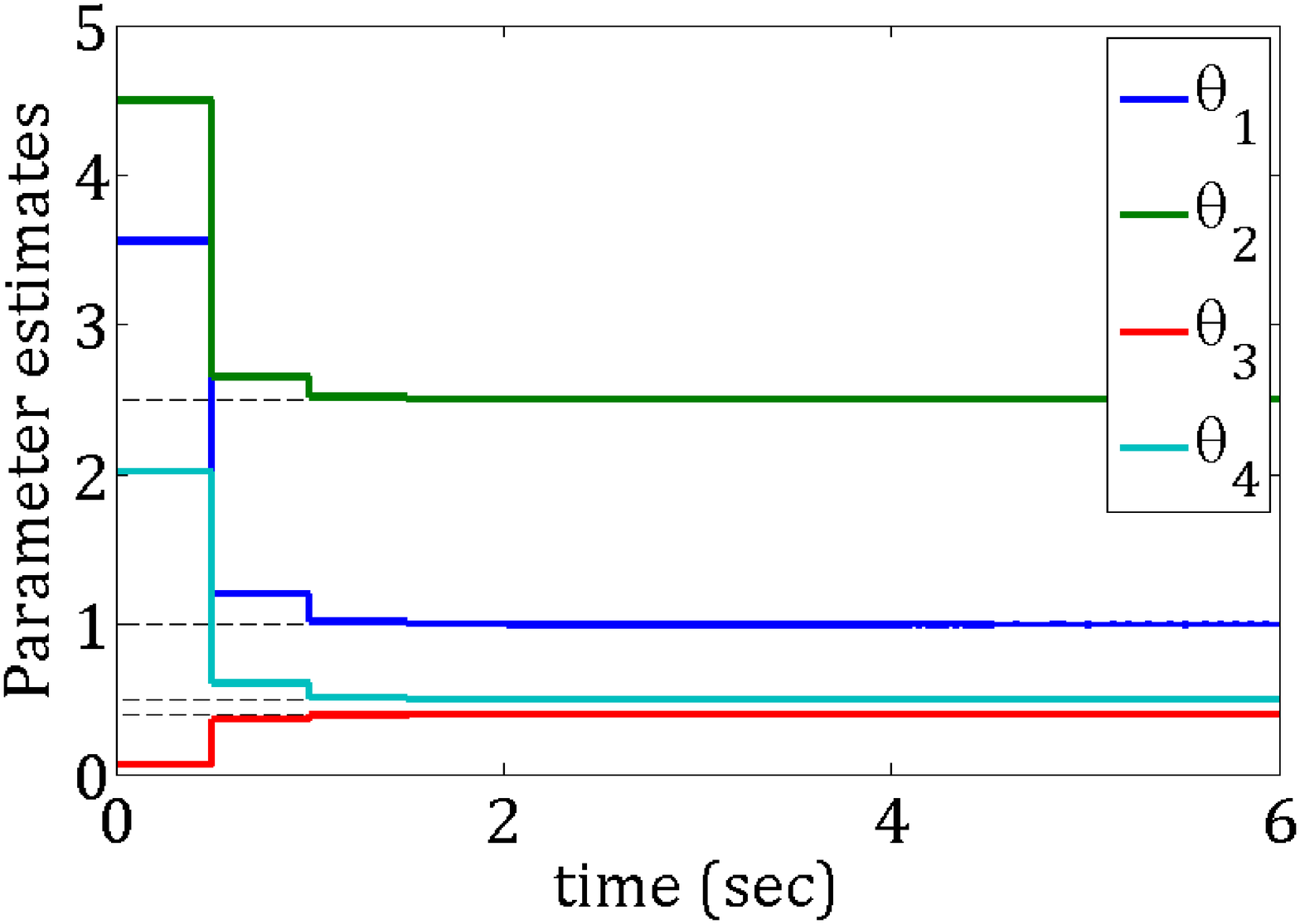}}
    \subfloat[][Constrained]{\includegraphics[width=7cm,height=4cm]{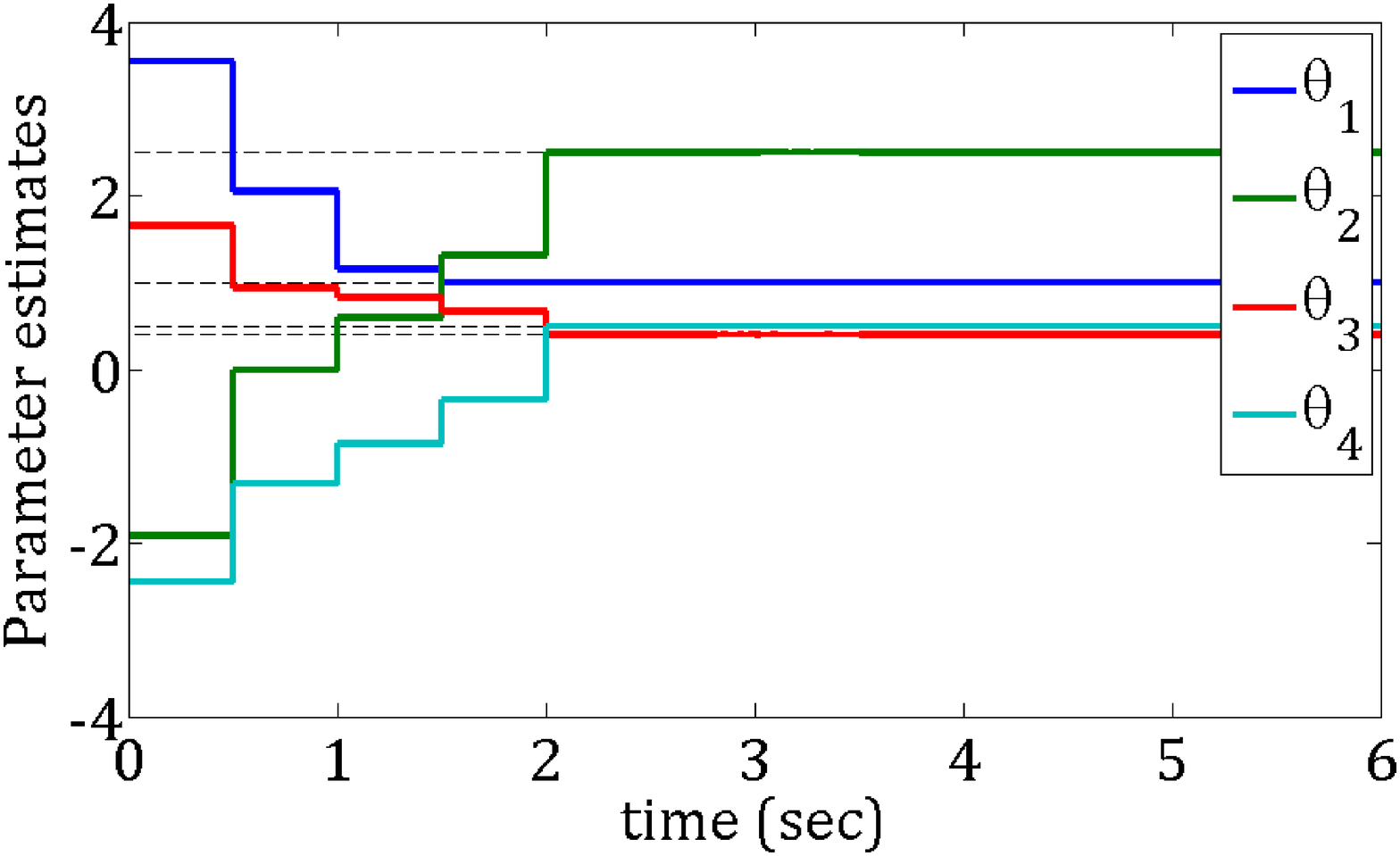}}
 \caption{Parameter estimate trajectory}
\label{fig_param_ex2}
\end{figure}

\section{Conclusion}\label{Conclusion}
A direct adaptive control technique is presented for use, in conjunction with concurrent learning approach, within the framework of model predictive control. The presented control technique undermines the need to switch between online learning phase and control phase by modulating the control sequences and the parameter estimates simultaneously at each computation instant. Theoretical analysis shows that the concurrent learning based adaptive model predictive control system is asymptotically stable with asymptotic parameter convergence. Numerical simulation results validated the theoretical claims and also showed that parameter estimation error decreases with increasing number of LGR nodes. However, associated with increased number of LGR points is increased computational burden. Therefore, a trade off is necessary between computational burden and parameter estimation error.

In future, the effect of actuator dynamics will be considered. Also, other discretization methods will be considered. Candidate discretization methods are; the use of Laguerre functions, other collocation methods like Runge-Kutta, etc.
\section{Acknowledgment}
All thanks be to my Lord and Personal Savior Jesus Christ.

\begin{small}
\bibliographystyle{abbrv}
\bibliography{references}

\begin{thebibliography}{10}

\bibitem{adaptiveMPC1}
V.~Adetola, D.~DeHaan, and M.~Guay.
\newblock Adaptive model predictive control for constrained nonlinear systems.
\newblock {\em Systems \& Control Letters}, 58(5):320--326, 2009.

\bibitem{adaptiveMPC2}
A.~Aswani, H.~Gonzalez, S.~S. Sastry, and C.~Tomlin.
\newblock Provably safe and robust learning-based model predictive control.
\newblock {\em Automatica}, 2013.

\bibitem{pseudospectral3}
D.~Benson.
\newblock {\em A Gauss pseudospectral transcription for optimal control}.
\newblock PhD thesis, Massachusetts Institute of Technology, 2005.

\bibitem{adaptiveMPC3}
P.~Bouffard, A.~Aswani, and C.~Tomlin.
\newblock Learning-based model predictive control on a quadrotor: Onboard
  implementation and experimental results.
\newblock In {\em Robotics and Automation (ICRA), 2012 IEEE International
  Conference on}, pages 279--284. IEEE, 2012.

\bibitem{MPC5}
E.~F. Camacho and C.~Bordons.
\newblock {\em Model predictive control}, volume~2.
\newblock Springer London, 2004.

\bibitem{MPC10}
E.~Camponogara, D.~Jia, B.~H. Krogh, and S.~Talukdar.
\newblock Distributed model predictive control.
\newblock {\em Control Systems, IEEE}, 22(1):44--52, 2002.

\bibitem{canuto2006erratum}
C.~Canuto, M.~Y. Hussaini, A.~Quarteroni, and T.~A. Zang.
\newblock {\em Spectral Methods}.
\newblock Springer, 2006.

\bibitem{quasi_infinite_MPC2}
C.~Chen and L.~Shaw.
\newblock On receding horizon feedback control.
\newblock {\em Automatica}, 18(3):349--352, 1982.

\bibitem{chen1997stability}
H.~Chen.
\newblock {\em Stability and robustness considerations in nonlinear model
  predictive control}.
\newblock VDI-Verlag, 1997.

\bibitem{quasi_infinite_MPC1}
H.~Chen and F.~Allg{\"o}wer.
\newblock A quasi-infinite horizon nonlinear model predictive control scheme
  with guaranteed stability.
\newblock {\em Automatica}, 34(10):1205--1217, 1998.

\bibitem{Concurrent_Learning1}
G.~Chowdhary and E.~Johnson.
\newblock Concurrent learning for convergence in adaptive control without
  persistency of excitation.
\newblock In {\em Decision and Control (CDC), 2010 49th IEEE Conference on},
  pages 3674--3679. IEEE, 2010.

\bibitem{chowdhary2013concurrent}
G.~Chowdhary, M.~M{\"u}hlegg, J.~P. How, and F.~Holzapfel.
\newblock Concurrent learning adaptive model predictive control.
\newblock In {\em Advances in Aerospace Guidance, Navigation and Control},
  pages 29--47. Springer, 2013.

\bibitem{Concurrent_Learning4}
G.~Chowdhary, T.~Yucelen, M.~M{\"u}hlegg, and E.~N. Johnson.
\newblock Concurrent learning adaptive control of linear systems with
  exponentially convergent bounds.
\newblock {\em International Journal of Adaptive Control and Signal
  Processing}, 2012.

\bibitem{Concurrent_Learning2}
G.~V. Chowdhary and E.~N. Johnson.
\newblock Theory and flight-test validation of a concurrent-learning adaptive
  controller.
\newblock {\em Journal of Guidance, Control, and Dynamics}, 34(2):592--607,
  2011.

\bibitem{MPC9}
B.~De~Schutter and T.~Van Den~Boom.
\newblock Model predictive control for max-plus-linear discrete event systems.
\newblock {\em Automatica}, 37(7):1049--1056, 2001.

\bibitem{duarte1989combined}
M.~A. Duarte and K.~S. Narendra.
\newblock Combined direct and indirect approach to adaptive control.
\newblock {\em Automatic Control, IEEE Transactions on}, 34(10):1071--1075,
  1989.

\bibitem{pseudospectral1}
G.~Elnagar, M.~A. Kazemi, and M.~Razzaghi.
\newblock The pseudospectral legendre method for discretizing optimal control
  problems.
\newblock {\em Automatic Control, IEEE Transactions on}, 40(10):1793--1796,
  1995.

\bibitem{pseudospectral2}
G.~N. Elnagar and M.~Razzaghi.
\newblock Short communication: A collocation-type method for linear quadratic
  optimal control problems.
\newblock {\em Optimal Control Applications and Methods}, 18(3):227--235, 1997.

\bibitem{pseudospectral_infinite1}
F.~Fahroo and I.~M. Ross.
\newblock Pseudospectral methods for infinite-horizon nonlinear optimal control
  problems.
\newblock {\em Journal of Guidance, Control, and Dynamics}, 31(4):927--936,
  2008.

\bibitem{adaptiveMPC4}
H.~Fukushima, T.-H. Kim, and T.~Sugie.
\newblock Adaptive model predictive control for a class of constrained linear
  systems based on the comparison model.
\newblock {\em Automatica}, 43(2):301--308, 2007.

\bibitem{pseudospectral_infinite2}
D.~Garg, W.~W. Hager, and A.~V. Rao.
\newblock Pseudospectral methods for solving infinite-horizon optimal control
  problems.
\newblock {\em Automatica}, 47(4):829--837, 2011.

\bibitem{pseudospectral5}
D.~Garg, M.~A. Patterson, C.~Francolin, C.~L. Darby, G.~T. Huntington, W.~W.
  Hager, and A.~V. Rao.
\newblock Direct trajectory optimization and costate estimation of
  finite-horizon and infinite-horizon optimal control problems using a radau
  pseudospectral method.
\newblock {\em Computational Optimization and Applications}, 49(2):335--358,
  2011.

\bibitem{hildebrand1987introduction}
F.~B. Hildebrand.
\newblock {\em Introduction to numerical analysis}.
\newblock Courier Dover Publications, 1987.

\bibitem{justin2011convergence}
R.~Justin, Z.~Anatoly, and J.-S. Li.
\newblock Convergence of a pseudospectral method for optimal control of complex
  dynamical systems.
\newblock In {\em Decision and Control, 2011 IEEE Conference on}, pages
  5553--5558. IEEE, 2012.

\bibitem{quasi_infinite_MPC3}
S.~a. Keerthi and E.~G. Gilbert.
\newblock Optimal infinite-horizon feedback laws for a general class of
  constrained discrete-time systems: Stability and moving-horizon
  approximations.
\newblock {\em Journal of optimization theory and applications},
  57(2):265--293, 1988.

\bibitem{khalil_book}
H.~K. Khalil.
\newblock {\em Nonlinear systems}, volume~3.
\newblock Prentice hall Upper Saddle River, 2002.

\bibitem{lavretsky2009combined}
E.~Lavretsky.
\newblock Combined/composite model reference adaptive control.
\newblock {\em Automatic Control, IEEE Transactions on}, 54(11):2692--2697,
  2009.

\bibitem{mayne2000nonlinear}
D.~Mayne.
\newblock Nonlinear model predictive control: Challenges and opportunities.
\newblock In {\em Nonlinear model predictive control}, pages 23--44. Springer,
  2000.

\bibitem{quasi_infinite_MPC4}
D.~Q. Mayne and H.~Michalska.
\newblock Receding horizon control of nonlinear systems.
\newblock {\em Automatic Control, IEEE Transactions on}, 35(7):814--824, 1990.

\bibitem{MPC8}
S.~J. Qin and T.~A. Badgwell.
\newblock A survey of industrial model predictive control technology.
\newblock {\em Control engineering practice}, 11(7):733--764, 2003.

\bibitem{MPC6}
T.~J. van~den Boom and T.~Backx.
\newblock Model predictive control.
\newblock {\em Lecture Notes for the Dutch Institute of Systems and Control,
  Winterterm}, 2004, 2003.

\bibitem{pseudospectral4}
P.~Williams.
\newblock Jacobi pseudospectral method for solving optimal control problems.
\newblock {\em Journal of Guidance, Control, and Dynamics}, 27(2):293--297,
  2004.

\end{thebibliography}
\end{small}
\end{document}